\pdfoutput=1
\documentclass[11pt,reqno]{amsart}

\usepackage{amsmath,amsthm,amssymb,amsfonts}
\usepackage{enumerate}
\usepackage{tabto}
\usepackage[mathscr]{euscript}
\usepackage{xcolor}
\usepackage{layout}
\usepackage{fancyhdr}
\usepackage{array}
\usepackage{mathtools}
\usepackage{graphicx}
\usepackage{bm}
\usepackage{enumitem}
\usepackage{caption} 
\usepackage{color}
\usepackage{mathrsfs}
\usepackage{csquotes}
\usepackage{bookmark}
\usepackage{float}
\usepackage{multirow}
\usepackage[square,numbers,sort&compress]{natbib}
\usepackage{hyperref}
\hypersetup{colorlinks=true,linkcolor=blue,citecolor=red}
\allowdisplaybreaks
\usepackage{soul}
\usepackage{upgreek}

\def\Xint#1{\mathchoice
{\XXint\displaystyle\textstyle{#1}}%
{\XXint\textstyle\scriptstyle{#1}}%
{\XXint\scriptstyle\scriptscriptstyle{#1}}%
{\XXint\scriptscriptstyle\scriptscriptstyle{#1}}%
\!\int}
\def\XXint#1#2#3{{\setbox0=\hbox{$#1{#2#3}{\int}$ }
\vcenter{\hbox{$#2#3$ }}\kern-.6\wd0}}

\def\dashint{\Xint-}

\makeatletter
\renewcommand{\subsection}{%
  \@startsection{subsection}{2}{\z@}%
  {-3.25ex\@plus -1ex\@minus -.2ex}%
  {1.5ex\@plus .2ex}%
  {\normalfont\bfseries}}
\makeatother

\newtheorem{theorem}{Theorem}[section]
\newtheorem{lemma}[theorem]{Lemma}

\newtheorem{remark}[theorem]{Remark}
\theoremstyle{definition}
\newtheorem{definition}[theorem]{Definition}
\numberwithin{equation}{section}

\newcommand{ \dist }{ \operatorname{dist} }

\newcommand{ \loc }{ \operatorname{loc} }

\newcommand{\N}{\mathcal{N}}
\newcommand{\A}{\mathcal{A}}
\newcommand{\B}{\mathcal{B}}
\newcommand{\ve}{\varepsilon}

\newcommand{\ga} {\gamma}

\newcommand{\om} {\Omega}

\newcommand{\noi} {\noindent}

\newcommand{\ds} {\displaystyle}

\newcommand{\wdt} {\widetilde}

\newcommand{\W}{W^{1,\Psi}}

\newcommand{\nonum} {\nonumber}

\begin{document}
\title[Gradient estimates for two modulating coefficients]{Gradient estimates for generalized double phase problems with two modulating coefficients}

\author{Jehan Oh}\address{Department of Mathematics, Kyungpook National University, Daegu, 41566, Republic of Korea} \email{jehan.oh@knu.ac.kr}
\author{Ambesh Kumar Pandey}\address{Department of Applied Mathematics, Kongju National University, Gongju-si, Chungcheongnam-do, 32588, Republic of Korea} \email{pandey.ambesh190@gmail.com}

\date{\today}
\thanks{This work is supported by National Research Foundation of Korea (NRF) grant funded by the Korea government [Grant Nos. RS-2025-00555316 and RS-2025-25415411].}

\begin{abstract}
We prove local Calder\'on-Zygmund estimates for distributional solutions to non-uniformly elliptic equations in divergence form modeled on the energy density $a(x)G(|z|)+b(x)H(|z|)$, where $G$ and $H$ are Young functions and $a(\cdot)$, $b(\cdot)$ are nonnegative coefficients, H\"older continuous with exponents $\alpha$ and $\beta$. Either coefficient may vanish, and only their sum is bounded away from zero, so that each phase can dominate or disappear. Under a gap condition that reduces to $q/p\le1+\min\{\alpha,\beta\}/n$ for $G(t)=t^p$ and $H(t)=t^q$, including the borderline case of equality, we show that the energy density of the gradient inherits the Orlicz integrability of the energy density of the datum. The proof combines a fractional differentiability estimate, which gives higher integrability up to the borderline growth, with comparison arguments that freeze the two coefficients in the appropriate order.
\end{abstract}
\maketitle
\tableofcontents

\section{Introduction}
\noi In recent years, Calder\'on-Zygmund estimates for quasilinear elliptic equations with non-standard growth have received considerable attention. In particular, double phase problems have been extensively studied as models for heterogeneous materials with two different hardening exponents. This naturally motivates the study of more general models involving two modulating coefficients. 

In this article, we are concerned with the study of local Calder\'on-Zygmund estimates for distributional solutions to a class of non-uniformly elliptic equations governed by \emph{two modulating coefficients}. This extends the usual double phase setting, where a single modulating weight interpolates between two growth rates to a two parameter anisotropic structure. The model equation we consider is given by
\begin{equation}\label{model_problem_1}
\begin{aligned}
-\operatorname{div}&\bigg(a(x)G'(|Du|)\frac{Du}{|Du|}
      +b(x)H'(|Du|)\frac{Du}{|Du|}\bigg) \\[3pt]
&\qquad=-\operatorname{div}\bigg(a(x)G'(|F|)\frac{F}{|F|}
      +b(x)H'(|F|)\frac{F}{|F|}\bigg) \quad \text{in } \Omega,
\end{aligned}
\end{equation}
where $F:\Omega\to\mathbb{R}^n$ is a given vector field and
$\Omega\subset\mathbb{R}^n$ is a bounded open domain with $n\ge 2$.

Throughout, we assume $G,H:[0,\infty)\to[0,\infty)$ are Young functions in the sense of Definition~\ref{def:Young} in Section~\ref{sec_prelim} and that the two coefficient functions
$a,b:\Omega\to\mathbb{R}$ satisfy
\begin{equation} 
0 \le a(\cdot) \in C^{0,\alpha}(\Omega), \qquad
0 \le b(\cdot) \in C^{0,\beta}(\Omega), \qquad
\alpha,\beta \in (0,1],
\label{coeff_1}
\end{equation}
and there exists $\mu>0$ such that
\begin{equation}\label{coeff_2}
   a(x)+b(x)\ge\mu
\end{equation}
for all $x\in \Omega$, which avoids the case when $a(\cdot)\equiv0\equiv b(\cdot)$.
Since the problem involves two modulating coefficients rather than a single coefficient, in order to study the
interaction between different phases, we impose the following gap condition
\begin{equation}\label{condition_1}
G \prec H,
\end{equation}
together with the assumption
\begin{equation}
\kappa:=\sup_{t>0} \frac{H(t)}{G(t)+G^{1+\frac{\gamma}{n}}(t)} < +\infty
\quad \text{ where } \gamma := \min\{\alpha,\beta\}.
\label{condition_2}
\end{equation}
\begin{remark}
    In the special case $G(t)=t^p$ and $H(t)=t^q$ with $q\ge p>1$, condition \eqref{condition_2} reduces to
    \begin{equation}\label{eq:growth_p_q}
        \frac qp\le1+\frac\ga n,
    \end{equation}
    and hence extends the setting of \cite{Kim2026}, where the strict inequality was assumed. Our results thus also cover the borderline case $q/p=1+\ga/n$, which requires special treatment.
\end{remark}
\begin{remark}
Assumption \eqref{condition_1} is the natural counterpart, in this Orlicz setting, of the growth restriction $p \le q$ familiar from the
classical double phase functional $|Du|^p + a(x)|Du|^q$. This gap condition plays an important role in establishing Lemma~\ref{lem:modular_equivalence}.
\end{remark}
\begin{remark}
The exponent $\gamma = \min\{\alpha,\beta\}$
in \eqref{condition_2} reflects the fact that the regularity available
for the two coefficient structure is governed by the coefficient with the smaller H\"older exponent, i.e., the estimates can only be as good as the less regular coefficient allows.  
\end{remark}

Our main interest is to determine sharp conditions on $G$, $H$, $\alpha$ and $\beta$ under which the local
Calder\'on-Zygmund type relation
\begin{multline}\label{eq:1.3prime}
a(x)G(|F|) + b(x)H(|F|) \in L^{\Theta}_{\mathrm{loc}}(\Omega) \\ \Longrightarrow\;
a(x)G(|Du|) + b(x)H(|Du|) \in L^{\Theta}_{\mathrm{loc}}(\Omega)
\end{multline}
holds true for every Young function $\Theta \in \mathcal{N}$ with
$s(\Theta) \ge 1$ (see Subsection~3.1 for the definitions of
$s(\cdot)$ and $\mathcal{N}$).

We remark that equation~\eqref{model_problem_1} is the Euler-Lagrange equation associated with the functional
\begin{equation}
v \;\mapsto\; \mathcal{J}(v,\Omega)
- \int_{\Omega} \left\langle a(x)G'(|F|)\frac{F}{|F|}
+ b(x)H'(|F|)\frac{F}{|F|},\, Dv \right\rangle dx,
\end{equation}
where the underlying energy $\mathcal{J}$ is the double phase functional given by
\begin{equation}\label{eq:general_energy}
\mathcal{J}(v,\Omega) = \int_{\Omega}
\Psi(x,Dv)\,dx,
\end{equation}
and the integrand is denoted by
\[\Psi(x,z):=a(x)G(|z|)+b(x)H(|z|)\]
for $x\in\Omega$ and $z\in\mathbb R^n$.

The functional $\mathcal{J}$ for $a(\cdot)\equiv1$ provides a natural generalization of the classical
double phase functional
$$W^{1,1}(\om)\ni v\mapsto \mathcal{P}(v,\om):=\int_{\om}\bigl(|Dv|^p+a(x)|Dv|^q\bigr)\,dx,
\qquad 1<p\le q.$$
The functional $\mathcal P$ was first introduced by Zhikov \cite{Zhikov1986,Zhikov1993,Zhikov1994} in the context of homogenization, as a model for strongly anisotropic materials. For two constituent materials with power hardening exponents $p$ and $q$, respectively, the coefficient $a(\cdot)$ dictates the local geometry of the composite, interpolating between the two phases as it varies across the domain, so that the associated Euler-Lagrange operator is non-uniformly elliptic, with ellipticity ratio degenerating precisely where $a(\cdot)$ vanishes. The functional also provides some of the first examples of the Lavrentiev phenomenon, whereby minimizers cannot, in general, be approximated by smooth functions, see \cite{Zhikov1995,Zhikov1997}. For the regularity theory, Colombo and Mingione, together with Baroni \cite{Baroni2015,Baroni2018,Colombo2015,Colombo2015a}, established sharp gradient H\"older continuity for minimizers of $\mathcal P$, proving that the necessary and sufficient condition between the two exponents and the H\"older regularity of the coefficient is
$$\frac{q}{p}\leq1+\frac{\alpha}{n}.$$
Since then, the regularity theory for double phase problems has developed considerably. Beck and Mingione \cite{Beck} studied non-uniformly elliptic variational problems and obtained optimal conditions for local Lipschitz regularity under general growth assumptions. Byun et al. \cite{Byun2017a} obtained global Calder\'on-Zygmund estimates for double phase problems with BMO coefficients on nonsmooth domains. More recently, \cite{Byun2026} established Calder\'on-Zygmund estimates for double phase problems with matrix weights. Gradient estimates and regularity results for multi-phase variational integrals were further developed in \cite{DeFilippis2019a,DeFilippis2022}. See also \cite{Ok2020} for regularity results under additional integrability assumptions on the minimizer.

The functional $\mathcal J$ in \eqref{eq:general_energy}, the generalized double phase functional with two modulating coefficients includes each of the following as a special case.
\begin{enumerate}
\item[$(1)$] $p$-Laplacian: $\Psi(x,z)=|z|^p$, e.g. \cite{Uhlenbeck,Uralceva1968,Manfredi1988}.
\item[$(2)$] Double phase: $\Psi(x,z)=|z|^p+a(x)|z|^q$,
e.g. \cite{Byun2018,Colombo2016, Colombo2015,Colombo2015a,Mingione_2019,Ok2017,Shin2020}.
\item[$(3)$] Double phase with two modulating coefficients: $\Psi(x,z)=a(x)|z|^p+b(x)|z|^q$, e.g. \cite{Kim2024a,Kim2025a,Kim2026}.
\item[$(4)$] Orlicz growth: $\Psi(x,z)=G(|z|)$, e.g. \cite{Baasandorj2021a,Byun2020a,Cho2018,Diening2012,harjulehto2019generalized}.
\item[$(5)$] Orlicz double phase: $\Psi(x,z)=G(|z|)+a(x)H(|z|)$, e.g. \cite{Baasandorj2020,Byun2020,Baasandorj_2025,Ok_2022}.
\end{enumerate}
The examples discussed above belong to the class of variational problems with non-standard growth, introduced by Marcellini in his pioneering works \cite{Marcellini1986,Marcellini1989,Marcellini1991}. Over the years, this class of problems has been the subject of extensive study, see for instance \cite{Breit,Esposito2004,Esposito1999,Fonseca2004,Schmidt} and the references therein. For recent developments in problems with non-standard growth and various types of non-uniform ellipticity, we refer to \cite{Mingione_2021}.

The functional $\mathcal{J}$ can be viewed as a generalized double phase energy describing a heterogeneous medium in which two phases, governed by the growth functions $G$ and $H$, coexist and compete throughout the domain. The coefficients $a(x)$ and $b(x)$ determine the local contribution of the growth functions $G$ and $H$, respectively. The assumption $a(x)+b(x)\ge\mu>0$ guarantees that at least one phase is active at every point, so that the energy remains non-degenerate. In regions where $a(x)$ is dominant, the $G$-phase contributes more significantly, while in regions where $b(x)$ is larger, the $H$-phase becomes dominant. In general, both phases may contribute simultaneously. Therefore, we can say that the coefficients $a(\cdot)$ and $b(\cdot)$ regulate the mixture between two different materials, whose constitutive behaviors are described by the growth functions $G$ and $H$, respectively. Unlike the classical double phase functional introduced by Zhikov, where one phase is always present and only the second phase is modulated by a spatially varying coefficient, here both phases are independently weighted by the coefficients $a(x)$ and $b(x)$, allowing either phase to dominate.

One of the most difficult parts in proving the relation
\eqref{eq:1.3prime} under the condition $\kappa<\infty$ is to obtain the
higher integrability estimate of the gradient up to the growth
$G^{1+\gamma/n}$ (see Theorem~\ref{thm:fractional_diff}) and to freeze the two coefficients $a$ and $b$ in the
correct order during the comparison process (see Section~\ref{sec_7}). We derive the higher
integrability estimate by proving a fractional differentiability result
and applying the fractional Sobolev embedding theorem, an approach
motivated by the pioneering works of Colombo and Mingione
\cite{Colombo2015,Colombo2016}. Since the $G$-term now carries the coefficient $a(\cdot)$ instead of the constant $1$, this argument does not carry over
directly from the single-coefficient case treated in \cite{Baasandorj2020}. Therefore, we split each ball into the region where $a$ is bounded below and the region where $b$ is, and pass between the two using the gap condition $G\prec H$. In addition, because of the lack of homogeneity properties for equations of the generalized double phase type, we adopt the maximal function-free technique, first
introduced by Acerbi and Mingione \cite{Mingione2007} and later used in several different settings. This technique is particularly well suited to problems with
non-standard growth.


This paper is organized as follows. Section~\ref{sec_2} states our main assumptions and results, including Theorem~\ref{main_thm}. Section~\ref{sec_prelim} collects background material on Young functions and the Orlicz and Musielak-Orlicz spaces used throughout.
Section~\ref{sec_4} discusses the Lavrentiev phenomenon and proves a
Sobolev-Poincar\'e inequality adapted to the two modulating coefficients.
In Section~\ref{sect_5}, we prove a fractional differentiability result for the
homogeneous equation. Section~\ref{sec_6} combines results from Section~\ref{sect_5} with a Gehring-type argument to obtain a higher integrability estimate. The last section proves the main result via an exit time argument based on a chain of suitable comparison functions.

\section{Main results}\label{sec_2}

\noi We study distributional solutions to the problem
\begin{equation}\label{model_problem_2}
-\operatorname{div}\A(x,Du)=-\operatorname{div}\B(x,F) \qquad\text{in }\Omega.
\end{equation}
The equation above is understood in a general framework and is modeled on the prototype equation \eqref{model_problem_1}. Let $\tilde\A:[0,\infty)\times[0,\infty)\times(\mathbb R^n\setminus\{0\})\to\mathbb R^n$ be a vector field continuously differentiable with respect to the gradient variable $z$, and set
$$\A(x,z):=\tilde\A(a(x),b(x),z).$$ 
We assume that for every $a,a_1,a_2,b,b_1,b_2\ge0$ and $z,\xi\in\mathbb R^n\setminus\{0\}$ there exist constants $0<\nu\le L<\infty,$ such that the following holds true
\begin{equation}\label{growth_conditions}
\left\{
\begin{aligned}
&|\tilde\A(a,b,z)|+|D_z\tilde\A(a,b,z)||z|\le L\,\frac{aG(|z|)+bH(|z|)}{|z|},\\[6pt]
&\nu\,\frac{aG(|z|)+bH(|z|)}{|z|^2}|\xi|^2\le\langle D_z\tilde\A(a,b,z)\xi,\xi\rangle,\\[6pt]
&\left|\tilde\A(a_1,b,z)-\tilde\A(a_2,b,z)\right|\le L\,\frac{|a_1-a_2|\,G(|z|)}{|z|},\\[6pt]
&\left|\tilde\A(a,b_1,z)-\tilde\A(a,b_2,z)\right|\le L\,\frac{|b_1-b_2|\,H(|z|)}{|z|}.
\end{aligned}
\right.
\end{equation}
Here $G$ and $H$ are Young functions satisfying assumptions \eqref{condition_1} and \eqref{condition_2}, while
$a(\cdot)$ and $b(\cdot)$ are as in \eqref{coeff_1} and \eqref{coeff_2}.
\begin{remark}
We note that the assumptions in \eqref{growth_conditions} imply that
\begin{equation*}
|\A(x_1,z)-\A(x_2,z)|\le L\,\frac{|a(x_1)-a(x_2)|G(|z|)+|b(x_1)-b(x_2)|H(|z|)}{|z|}.
\end{equation*}
\end{remark}
\begin{remark}
We can observe that the above assumptions are satisfied by the model operator arising from the double phase energy functional $\mathcal P(v,\Omega)=\int_\Omega\bigl(a(x)G(|Dv|)+b(x)H(|Dv|)\bigr)\,dx$, namely
\begin{equation*}
\mathcal A(a,b,z):=a\,G'(|z|)\frac{z}{|z|}+b\,H'(|z|)\frac{z}{|z|},\qquad a,b\ge0,\ z\in\mathbb R^n\setminus\{0\}.
\end{equation*}
\end{remark}

The vector field $\B:\Omega\times\mathbb R^n\rightarrow\mathbb R^n$
is assumed to satisfy the growth condition
\begin{equation}\label{growth_condition_2}
|\B(x,z)|\le L\,\frac{a(x)G(|z|)+b(x)H(|z|)}{|z|}
\end{equation}
for every
$x\in\Omega$
and
$z\in\mathbb R^n\setminus\{0\}$.

For the sake of convenience, we introduce the notations
\begin{equation}\label{function_psi}
\left.
\begin{aligned}
\Psi(x,z)&:=a(x)G(|z|)+b(x)H(|z|),\\
\Psi_1(x,z)&:=a(x)G(|z|)+b(x)H(|z|)+1
\end{aligned}\right.\qquad \text{ for }(x,z)\in\Omega\times\mathbb R^n.
\end{equation}

By abuse of notation, we shall also write $\Psi(x,z)$ for
$x\in\Omega$ and $z\in\mathbb R$ whenever no confusion may arise. Throughout the paper, the notation $\texttt{data}$ denotes the following collection of parameters
\small{\begin{equation*}
\begin{split}
\texttt{data}:=\bigl(n,\kappa,s(G),s(H),\nu,L,\alpha,\beta,&\mu,\|a\|_{L^\infty(\Omega)},[a]_{0,\alpha},\\
&\|b\|_{L^\infty(\Omega)},[b]_{0,\beta},\|\Psi_1(\cdot,Du)\|_{L^1(\Omega)}\bigr),
\end{split}
\end{equation*}}
where $s(G)$ and $s(H)$ denote the indices of $G$ and $H$, respectively (see Subsection~$3.1$ for their definition) and $\kappa$ is given by \eqref{condition_2}.
We now present the statement of the main result.
\begin{theorem}\label{main_thm}
Let $u\in W^{1,1}(\Omega)$ be a distributional solution to \eqref{model_problem_2}  under the assumptions \eqref{coeff_1}--\eqref{condition_2}, \eqref{growth_conditions} and \eqref{growth_condition_2}, with
\begin{equation}
\Psi(x,Du),\;\Psi(x,F)\in L^1(\Omega).
\end{equation}
Then for every $\Theta\in\mathcal{N}$ with $s(\Theta)\geq 1$,
\begin{equation}
\Psi(x,F)\in L^{\Theta}_{\loc}(\Omega)
\;\Longrightarrow\;
\Psi(x,Du)\in L^{\Theta}_{\loc}(\Omega).
\end{equation}
Moreover, for open subsets $\Omega_0\Subset\Omega_1\Subset\Omega$ with
\[
\dist(\Omega_0,\partial\Omega_1)\sim\dist(\Omega_1,\partial\Omega)\sim\dist(\Omega_0,\partial\Omega),
\]
there exist $r>0$ and $c>0$ depending on $\textnormal{\texttt{data}}$, $\dist(\Omega_0,\partial\Omega)$, $\|\Theta(\Psi(\cdot,F))\|_{L^1(\Omega_1)}$ and $s(\Theta)$ such that
\begin{equation}
\dashint_{B_{R/2}}\Theta\bigl(\Psi(x,Du)\bigr)\,dx
\leq
c\,\Theta\!\left(\dashint_{B_R}\Psi(x,Du)\,dx\right)
+c\dashint_{B_R}\Theta\bigl(\Psi(x,F)\bigr)\,dx+c
\end{equation}
holds for every ball $B_R\subset\Omega_0$ with $R\leq r$.
\end{theorem}
Theorem~\ref{main_thm} is a natural generalization of several existing Calder\'on-Zygmund estimates in the double phase setting. In particular, the results in the seminal papers \cite{Colombo2016,Mingione_2019} consider the power growth case \(G(t)=t^p\), \(H(t)=t^q\), with \(a\equiv1\) and \(\Theta(t)=t^\gamma\), \(\gamma>1\). The generalized Orlicz setting with \(a\equiv1\) was studied in \cite{Baasandorj2020}, while \cite{Kim2026} treats the power growth case \(G(t)=t^p\), \(H(t)=t^q\) with two modulating coefficients under a strict inequality of the growth condition \eqref{eq:growth_p_q}. Our result considers the generalized double phase structure with two modulating coefficients, $a(\cdot)$ and $b(\cdot)$, and thus includes these settings as special cases. We also mention that, in the classical setting \(G(t)=t^p\), \(H(t)=t^q\), with \(a\equiv1\), the counterexamples in \cite{Esposito2004,Fonseca2004} show that the assumptions of Theorem~\ref{main_thm} are optimal and cannot be dropped. In particular, the result is false when \eqref{condition_2} is not satisfied.


We conclude this section with a remark concerning the regularity assumptions imposed on the coefficient functions.
\begin{remark}
The H\"older continuity assumption on $a(\cdot)$ is mainly used to establish the fractional differentiability estimate in Theorem~\ref{thm:fractional_diff}, in particular, the Nikolski space estimate. The Calder\'on-Zygmund estimate itself requires only a VMO (or sufficiently small BMO) assumption (see Step 3, Proof of Theorem~\ref{main_thm}), while continuity is sufficient to prove the absence of the Lavrentiev phenomenon (see Theorem~\ref{thm:Lavrentiev}).
\end{remark}

\section{Preliminaries}\label{sec_prelim}
\noi In this section, we introduce the notation and collect some preliminary results that will be used throughout the paper.
\subsection{Notations}
We begin by introducing the notation used throughout the paper.

For $x_0\in\mathbb{R}^n$ and $\rho>0$, we denote by
$$
B_\rho(x_0):=\{x\in\mathbb{R}^n:|x-x_0|<\rho\}
$$
the open ball centered at $x_0$ with radius $\rho$. Whenever the center is clear from the context, we simply write $B_\rho$ instead of $B_\rho(x_0)$.

Given a measurable set $E\subset\mathbb{R}^n$ with $0<|E|<\infty$ and an integrable function $g:E\to\mathbb{R}^k$, we write
$$\dashint_E g(x)\,dx:=\frac1{|E|}\int_E g(x)\,dx$$ 
for the integral average of $g$ over $E$.

For $\beta\in(0,1]$ and $g:E\to\mathbb{R}$, the H\"older seminorm of $g$ on $E$ is defined by
$$[g]_{0,\beta;E}:=\sup_{\substack{x,y\in E\\x\neq y}}\frac{|g(x)-g(y)|}{|x-y|^\beta}.
$$
When the underlying domain is understood, we simply write $[g]_{0,\beta}$.

Throughout the paper, $c$ denotes a positive constant whose value may change from line to line. Constants whose values remain fixed in a given argument will be denoted by $c_1,c_*, C$, etc. Whenever necessary, the dependence of a constant on relevant parameters will be indicated explicitly. For example, $c\equiv c(\mathrm{\texttt{data}})$
means that $c$ depends only on the quantities collected in $\texttt{data}$.

We first recall the definition of a Young function.
\begin{definition}\label{def:Young}
A Young function $\Phi:[0,\infty)\rightarrow[0,\infty)$ is an increasing convex function satisfying
    \begin{equation}
        \Phi(0)=0, \quad\ds\lim_{s\to\infty}\Phi(s)=\infty,\quad \ds\lim_{s\to0^+}\frac{\Phi(s)}{s}=0 \quad \text{and}\quad\ds\lim_{s\to\infty}\frac{\Phi(s)}{s}=\infty.
    \end{equation}    
We denote by $\mathcal{N}$ the family of all Young functions $\Phi:[0,\infty)\to[0,\infty)$ such that $\Phi\in C^1([0,\infty))\cap C^2((0,\infty))$ and satisfy
\begin{equation}\label{eq:Phi-structure}
\frac{1}{s(\Phi)}\le\frac{t\Phi''(t)}{\Phi'(t)}\le s(\Phi),\qquad t>0,
\end{equation}
for some constant $s(\Phi)\ge1$.
An immediate consequence of \eqref{eq:Phi-structure} is
\[1+\frac1{s(\Phi)}\le\frac{t\Phi'(t)}{\Phi(t)}\le1+s(\Phi),\]
and therefore
\begin{equation}\label{eq:Phi-growth}
\Phi(t)\approx t\Phi'(t)\approx t^2\Phi''(t),\qquad t>0,
\end{equation}
where the implicit constants depend only on $s(\Phi)$.
 \end{definition}
 
Associated with every Young function $\Phi$, we define the nonlinear map
\begin{equation}\label{vector_field}
 V_\Phi:\mathbb{R}^n\setminus\{0\}\to\mathbb{R}^n,\qquad V_\Phi(z):=
 \left(\frac{\Phi'(|z|)}{|z|}\right)^{1/2}z.
\end{equation}
The map $V_\Phi$ allows us to express the monotonicity properties of the vector field $\A(x,z)$ in a simplified form. We shall frequently use the following inequalities collected from \cite{Diening_2008}.
\begin{gather}
|V_\Phi(z_1)-V_\Phi(z_2)|^2\approx
\Phi''(|z_1|+|z_2|)|z_1-z_2|^2\approx\frac{\Phi'(|z_1|+|z_2|)}{|z_1|+|z_2|}|z_1-z_2|^2,
\label{eq:VPhi-est}\\[3pt]
\Phi(|z_1-z_2|)\lesssim\Phi(|z_1|+|z_2|)\frac{|z_1-z_2|}{|z_1|+|z_2|},\nonumber\\[3pt]
\left\langle\Phi'(|z_1|)\frac{z_1}{|z_1|}-\Phi'(|z_2|)\frac{z_2}{|z_2|},
\,z_1-z_2\right\rangle\approx|V_\Phi(z_1)-V_\Phi(z_2)|^2,\label{eq:Phi-mon}
\end{gather}
where the implicit constants depend only on $n$ and $s(\Phi)$.

Using \eqref{growth_conditions}$_2$ together with the definition \eqref{vector_field}, we obtain
\begin{equation}\label{eq:A-monotonicity}
a(x)|V_G(z_1)-V_G(z_2)|^2+b(x)|V_H(z_1)-V_H(z_2)|^2
\lesssim\left\langle \A(x,z_1)-\A(x,z_2),z_1-z_2\right\rangle,
\end{equation}
for every $x\in\Omega$ and $z_1,z_2\in\mathbb{R}^n\setminus\{0\}$.

Moreover, arguing as in \cite{Baasandorj2020} and using the growth assumptions
\eqref{growth_conditions}, we have
\begin{equation}\label{eq:A-difference}
|\A(x,z_1)-\A(x,z_2)|
\lesssim
\left(
a(x)\frac{G'(|z_1|+|z_2|)}{|z_1|+|z_2|}+b(x)\frac{H'(|z_1|+|z_2|)}{|z_1|+|z_2|}
\right)|z_1-z_2|,
\end{equation}
where the implicit constant depends only on
$n$, $s(G)$, $s(H)$ and $L$.
\subsection{Fractional Sobolev spaces}
Before proceeding further, we recall some important results concerning fractional Sobolev spaces that will be used throughout the paper. For further details, we refer the readers to \cite{hitchhikers_guide, fractional_leoni,Colombo2015, Esposito2004} and the references therein.

For an open set $\Omega\subset\mathbb{R}^n$ and $h\in\mathbb{R}^n$, we define
$$\Omega_{|h|}:=\{x\in\Omega:\operatorname{dist}(x,\partial\Omega)>|h|\}.$$
Given a vector-valued function
$f:\Omega\to\mathbb{R}^N$, we define the finite difference operator
$\tau_h:L^1(\Omega;\mathbb{R}^N)\to L^1(\Omega_{|h|};\mathbb{R}^N)$ by
$$\tau_hf(x):=f(x+h)-f(x).$$

The following estimate for finite differences can be found in \cite[Lemma~3.1]{Baasandorj2020}.
\begin{lemma}\label{lem:finite-difference}
Let $B_\rho\subset B_R$ be concentric balls and let $\Phi\in\N$ be a Young function. Then
\begin{equation*}
\int_{B_\rho}
\Phi\!\left(\frac{|\tau_hf|}{|h|}\right)\,dx\le\int_{B_R}
\Phi(|Df|)\,dx,
\end{equation*}
whenever $f\in W^{1,\Phi}(B_R)$ and $h\in\mathbb{R}^n$ satisfies $|h|\le R-\rho$.
\end{lemma}
We recall the following embedding result for Nikolski spaces from \cite[Lemma~2.7]{Colombo2015a} and \cite[Lemma~3]{Esposito2004}, which will be used in the proof of the fractional differentiability result in Theorem~\ref{thm:fractional_diff}.
\begin{lemma}\label{lem:nikolski}
Let $f\in L^2(B_R)$ with $0<R\le1$. Assume that there exist constants
$\rho\in(0,R)$, $d\in(0,1)$ and $M>0$ such that, for every cut-off function
$\eta\in C_0^\infty(B_{\frac{\rho+R}{2}})$ satisfying
$$\chi_{B_\rho}\le\eta\le\chi_{B_{\frac{\rho+R}{2}}},
\qquad|D\eta|\le\frac{4}{R-\rho},$$
the estimate
\begin{equation*}
\int_{B_R}\eta^2|\tau_hf|^2\,dx\le M^2|h|^{2d}
\end{equation*}
holds for every $h\in\mathbb{R}^n$ with $|h|\le\frac{R-\rho}{4}.$
Then
$$f\in W^{\beta,2}(B_\rho,\mathbb{R}^k)\cap L^{\frac{2n}{n-2\beta}}(B_\rho,\mathbb{R}^k)
$$
for every $\beta\in(0,d)$. Moreover,
\begin{equation*}
\|f\|_{L^{\frac{2n}{n-2\beta}}(B_\rho)}\le c(R-\rho)^{-(2\beta+2d+2)}
\left(M+\|f\|_{L^2(B_R)}\right),
\end{equation*}
where $c=c(n,k,\beta,d)>0$.
\end{lemma}

\subsection{Orlicz-Sobolev spaces and Musielak-Orlicz-Sobolev spaces}
In this subsection, we recall some basic notions and results from the theory of Orlicz and Musielak-Orlicz spaces that will be used throughout the paper. We begin with the following partial order relation between Young functions (see \cite{Byun2020}).
\begin{definition}
Let $\Phi_1$ and $\Phi_2$ be Young functions. We write $$
\Phi_1 \prec \Phi_2$$
if the composition $
\Phi_2\circ \Phi_1^{-1}$
is a Young function.
\end{definition}
We will use the following consequence of the relation $\prec$ from \cite[Lemma~2.5]{Byun2020}.
\begin{lemma}
Let $\Phi_1$ and $\Phi_2$ be Young functions satisfying
$\Phi_1 \prec \Phi_2.$ Then
\[
\Phi_1(t)\leq\frac{1}{(\Phi_2\circ\Phi_1^{-1})(1)}\,\Phi_2(t),\qquad
\text{for all }t\geq \Phi_1^{-1}(1).
\]
\end{lemma}

\begin{definition}
Let $\Phi$ be a Young function.

\begin{enumerate}
\item
We say that $\Phi$ satisfies the $\Delta_2$-condition and write
$\Phi\in\Delta_2$, if there exists a constant
\mbox{$\Delta_2(\Phi)>0$} such that
$\Phi(2t)\le \Delta_2(\Phi)\Phi(t)$ for all $ t\ge0.$

\item We say that $\Phi$ satisfies the $\nabla_2$-condition and write
$\Phi\in\nabla_2$, if there exists a constant
\mbox{$\nabla_2(\Phi)>1$} satisfying
$\Phi(\nabla_2(\Phi)t)\ge 2\nabla_2(\Phi)\Phi(t)$ for all $t\ge0.$

\item If both conditions hold, then we simply write
$\Phi\in\Delta_2\cap\nabla_2.$
\end{enumerate}

The complementary Young function associated with $\Phi$ is defined by
$$\Phi^*(t)=
\sup_{s\ge0}\{st-\Phi(s)\},
\qquad t\ge0.$$

We observe that $\Phi^*$ is again a Young function and satisfies
$(\Phi^*)^*=\Phi$. Moreover,
$\Phi\in\nabla_2$
\text{ if and only if }
$\Phi^*\in\Delta_2
\text{ with }
2\nabla_2(\Phi)=\Delta_2(\Phi^*).$
The complementary pair $(\Phi,\Phi^*)$ also satisfies the following Young's inequality,
$$
st\le\Phi(s)+\Phi^*(t),
\qquad s,t\ge0.
$$
\end{definition}

When $\Phi$ satisfies both the $\Delta_2$ and $\nabla_2$ conditions, one can introduce the corresponding Orlicz and Orlicz--Sobolev spaces.

\begin{definition}
Let $\Omega\subset\mathbb R^n$ be a bounded domain and let
$n\ge1$. The Orlicz space associated with $\Phi$ is
$$
L^\Phi(\Omega;\mathbb R^N)
:=
\left\{
v\in L^1(\Omega;\mathbb R^N):
\int_\Omega\Phi(|v|)\,dx<\infty
\right\}.
$$

Equipped with the Luxemburg norm
$$
\|v\|_{L^\Phi(\Omega;\mathbb R^N)}
:=
\inf\left\{
\lambda>0:
\int_\Omega
\Phi\!\left(\frac{|v|}{\lambda}\right)\,dx
\le1
\right\},
$$
the space $L^\Phi(\Omega;\mathbb R^N)$ is a Banach space.
\end{definition}
It is well known that $
L^{\Phi^*}(\Omega;\mathbb{R}^N)
=
\left(L^\Phi(\Omega;\mathbb{R}^N)\right)',$ 
is the dual space of $
L^{\Phi}(\Omega;\mathbb{R}^N)$. In the scalar case $(N=1)$, we simply write
$L^\Phi(\Omega)$ instead of $L^\Phi(\Omega;\mathbb{R})$.

The associated Orlicz-Sobolev space is defined by
\begin{equation}
W^{1,\Phi}(\Omega;\mathbb{R}^N)
:=
\left\{
v\in L^\Phi(\Omega;\mathbb{R}^N):
Dv\in L^\Phi(\Omega;\mathbb{R}^{N\times n})
\right\},
\end{equation}
where $Dv$ denotes the distributional gradient of $v$. This space becomes a Banach space when equipped with the Luxemburg norm
\begin{equation}
\|v\|_{W^{1,\Phi}(\Omega;\mathbb{R}^N)}
:=
\|v\|_{L^\Phi(\Omega;\mathbb{R}^N)}
+
\|Dv\|_{L^\Phi(\Omega;\mathbb{R}^{N\times n})}.
\end{equation}

For $N=1$, we abbreviate $W^{1,\Phi}(\Omega):=W^{1,\Phi}(\Omega;\mathbb{R}).$
Furthermore,
$W^{1,\Phi}_0(\Omega;\mathbb{R}^N)$ is the closure of
$C_0^\infty(\Omega;\mathbb{R}^N)$ in
$W^{1,\Phi}(\Omega;\mathbb{R}^N)$ and
$W^{1,\Phi}_{\mathrm{loc}}(\Omega;\mathbb{R}^N)$ is defined in the usual way.

The following elementary properties from \cite[Lemma~2.8]{Byun2020} will be used throughout the paper.

\begin{lemma}
\label{lem:basic-N}
Let $\Phi\in\mathcal N$. Then the following assertions hold.

\begin{enumerate}
\item[$(1)$]
$\Phi\in\Delta_2\cap\nabla_2$ and the constants
$\Delta_2(\Phi)$ and $\nabla_2(\Phi)$ depend only on $s(\Phi)$.

\item[$(2)$]
For $\Lambda\ge1$ and $t\ge0$, we have $\Phi(\Lambda t)\le\Lambda^{\,s(\Phi)+1}\Phi(t).$

\item[$(3)$]
For $0<\lambda\le1$ and $t\ge0$, we have
$\Phi(\lambda t)\le\lambda^{\,1+\frac1{s(\Phi)}}\Phi(t).$
\end{enumerate}
\end{lemma}

We also use the following properties of the class $\mathcal N$ from \cite[Lemma~3.4]{Baasandorj2020}.
\begin{lemma}
\label{lem:operations-N}
Let $\Phi,\widetilde{\Phi}\in\mathcal N$. Then the following properties hold.

\begin{enumerate}
\item[$(1)$] For every $a,b>0$, the functions
$a\Phi+b\widetilde{\Phi}$ and
$\Phi\widetilde{\Phi}$ belong to $\mathcal N$. Moreover,
$s(\Phi+\widetilde{\Phi})=s(\Phi)+s(\widetilde{\Phi})$ and
$s(\Phi\widetilde{\Phi})=4\,s(\Phi)s(\widetilde{\Phi})
\bigl(s(\Phi)+s(\widetilde{\Phi})\bigr).$
\item[$(2)$] For every integer $m\ge1$, $\Phi^m\in\mathcal N$ and $s(\Phi^m)=s(\Phi)+(m-1)\bigl(s(\Phi)+1\bigr).$
\item[$(3)$] For every $\mu\ge0$, define $\Phi_\mu(t):=t^\mu\Phi(t).$
Then $\Phi_\mu\in\mathcal N$ and
$s(\Phi_\mu)=\mu+3[s(\Phi)]^2.$
\item[$(4)$] There exists $\alpha_0\in(0,1)$, which depends only on $s(\Phi)$, such that $\Phi^\alpha\in\mathcal N$
for every $\alpha\in(\alpha_0,1]$, where
$s(\Phi^\alpha)$ depends only on $s(\Phi)$ and $\alpha$.
\end{enumerate}
\end{lemma}

We also recall a Young type inequality for functions in $\mathcal N$ from \cite[Lemma~3.5]{Baasandorj2020}.
\begin{lemma}
\label{lem:Young epsilon}
For every $\Phi\in\mathcal N$, there exists a constant
$c=c(s(\Phi))>0$ such that
\begin{equation}
t\Phi'(s)+s\Phi'(t)
\le
\varepsilon\Phi(t)
+
\frac{c}{\varepsilon^{\,s(\Phi)}}\Phi(s),
\end{equation}
whenever $s,t\ge0$ and $0<\varepsilon\le1$.
\end{lemma}

\begin{lemma}
\label{lem:modular_equivalence}
Let $G$ and $H$ be two Young functions satisfying $G \prec H$. Let $\Omega \subset \mathbb{R}^n$,
$n \geq 2$, be an open bounded set and let $a, b :\Omega \to [0,\infty)$ be measurable functions as in \eqref{coeff_1} and \eqref{coeff_2}. Then there exist constants $C_1, C_2 > 0$, depending only on $\mu$,
$\|a\|_{L^{\infty}(\Omega)}$, $\|b\|_{L^{\infty}(\Omega)}$, $G$ and $H$,
such that
\begin{equation}\label{eq:equivalence}
    C_1\, G(|z|) \;\leq\; \Psi_1(x,z) \;\leq\; C_2\!\left(H(|z|) + 1\right),
\end{equation}
for every $x \in \Omega$ and every $z \in \mathbb{R}^n$.
\end{lemma}

\begin{proof}
Set $a_0 := \|a\|_{L^{\infty}(\Omega)}$ and $b_0 :=
\|b\|_{L^{\infty}(\Omega)}$.

\medskip
\noindent\textbf{Upper bound.}
Since $G \prec H$, there exists a constant $c > 0$ such that
\begin{equation}
\label{eq:G_leq_H}
    G(t) \leq c\!\left(H(t) + 1\right) \quad \forall\, t \geq 0,
\end{equation}
see \cite[Corollary~2.6]{Byun2020}. Using \eqref{eq:G_leq_H} together with
$a(x) \leq a_0$ and $b(x) \leq b_0$, we estimate
\begin{align*}
    \Psi_1(x,z)
    &= a(x)\,G(|z|) + b(x)\,H(|z|) + 1 \\
    &\leq a_0\,c\!\left(H(|z|)+1\right)
        + b_0\!\left(H(|z|)+1\right)
        + \left(H(|z|)+1\right) \\
    &= \left(c\,a_0 + b_0 + 1\right)\!\left(H(|z|) + 1\right).
\end{align*}
Setting $C_2 := c\,a_0 + b_0 + 1$ yields the upper bound in
\eqref{eq:equivalence}.

\medskip
\noindent\textbf{Lower bound.}
We prove that $C_1\,G(|z|) \leq \Psi_1(x,z)$ holds for a suitable
$C_1 > 0$ to be determined. We consider two cases according to the
size of $G(|z|)$.

\medskip
\noindent\textit{Case} $(i)$: $G(|z|) < 1$.

By definition \eqref{function_psi}, $\Psi_1(x,z) \geq 1$ for every
$x \in \Omega$ and $z \in \mathbb{R}^n$. Therefore,
\[
    \Psi_1(x,z) \;\geq\; 1 \;>\; G(|z|),
\]
so that $C_1\,G(|z|) \leq \Psi_1(x,z)$ holds for any choice of
$C_1 \leq 1$.

\medskip
\noindent\textit{Case} $(ii)$: $G(|z|) \geq 1$.

Since $H \circ G^{-1}$ is a Young function by assumption, the map
$t \mapsto (H \circ G^{-1})(t)/t$ is non-decreasing. Evaluating at
$t = G(|z|) \geq 1$ and comparing with $t = 1$ gives
\[
    \frac{(H \circ G^{-1})(G(|z|))}{G(|z|)}
    \;\geq\;
    (H \circ G^{-1})(1),
\]
which, recalling $(H \circ G^{-1})(G(|z|)) = H(|z|)$, yields
\begin{equation}
\label{eq:H_G_comparison}
    H(|z|) \;\geq\; (H \circ G^{-1})(1)\cdot G(|z|)
    \quad \text{whenever } G(|z|) \geq 1.
\end{equation}
Set $\lambda_0 \;:=\; \min\!\left\{1,\;(H \circ G^{-1})(1)\right\} \;>\; 0.$
Applying \eqref{eq:H_G_comparison}, $a(x), b(x) \geq 0$ and
then dropping the $+1$ term in \eqref{function_psi}, we obtain
\begin{align*}
    \Psi_1(x,z)
    &\geq a(x)\,G(|z|) + b(x)\,H(|z|) \\[4pt]
    &\geq a(x)\,G(|z|)
        + (H \circ G^{-1})(1)\cdot b(x)\,G(|z|) \\[4pt]
    &\geq \lambda_0\,a(x)\,G(|z|) + \lambda_0\,b(x)\,G(|z|) \\[4pt]
    &= \lambda_0\!\left(a(x) + b(x)\right)G(|z|) \\[4pt]
    &\geq \lambda_0\,\mu\,G(|z|),
\end{align*}
where in the third line we used $1 \geq \lambda_0$ and
$(H \circ G^{-1})(1) \geq \lambda_0$ and in the last line we
used assumption \eqref{coeff_2}.

Finally choosing the constant $C_1:=\min\!\left\{1,\;\mu\,\min\!\left\{1,\;(H \circ G^{-1})(1)\right\}\right\} \;>\; 0,$
the bound $C_1\,G(|z|) \leq \Psi_1(x,z)$ holds in both cases, completing
the proof of \eqref{eq:equivalence}.
\end{proof}

We briefly recall the notions from the theory of Musielak-Orlicz spaces that will be used throughout the paper. A function $\Psi:\Omega\times[0,\infty)\to[0,\infty)$ is called a Musielak-Orlicz function if the following two conditions hold.

\begin{enumerate}
\item For every $x\in\Omega$, the mapping $t\mapsto\Psi(x,t)$ is a Young function.
\item For every $t\ge0$, the mapping $x\mapsto\Psi(x,t)$ is measurable.
\end{enumerate}

\begin{definition}
Let $\Psi:\Omega\times[0,\infty)\to[0,\infty)$ be a Musielak-Orlicz function.

\begin{enumerate}
\item
We say that $\Psi$ satisfies the $\Delta_2$-condition if there exists a constant
\mbox{$\Delta_2(\Psi)>0$} such that
\begin{equation}
\Psi(x,2t)
\le
\Delta_2(\Psi)\Psi(x,t),
\quad\text{ for all }
x\in\Omega,\; t\ge0.
\end{equation}

\item
We say that $\Psi$ satisfies the $\nabla_2$-condition if there exists a constant
\mbox{$\nabla_2(\Psi)>1$} satisfying
\begin{equation}
\Psi\!\left(x,\nabla_2(\Psi)t\right)
\ge
2\nabla_2(\Psi)\Psi(x,t),
\quad
\text{ for all }x\in\Omega,\; t\ge0.
\end{equation}

\item
If both conditions are fulfilled, we write $\Psi\in\Delta_2\cap\nabla_2.$
\end{enumerate}
\end{definition}

The complementary Musielak-Orlicz function associated with $\Psi$ is defined by
\begin{equation}
\Psi^*(x,t):=\sup_{s\ge0}\{st-\Psi(x,s)\},
\qquad (x,t)\in\Omega\times[0,\infty).
\end{equation}
The function $\Psi^*$ is again a Musielak--Orlicz function and satisfies $(\Psi^*)^*=\Psi.$
Moreover, $\Psi\in\nabla_2
\quad\Longleftrightarrow\quad
\Psi^*\in\Delta_2,$
and $2\nabla_2(\Psi)=\Delta_2(\Psi^*).$

The corresponding Musielak-Orlicz class is
\begin{equation}
K^\Psi(\Omega;\mathbb R^N)
:=
\left\{
v:\Omega\to\mathbb R^N:
\int_\Omega\Psi(x,|v|)\,dx<\infty
\right\}.
\end{equation}

The Musielak-Orlicz space is defined by $L^\Psi(\Omega;\mathbb R^N)
:=\operatorname{span}K^\Psi(\Omega;\mathbb R^N).$
If $\Psi$ satisfies the $\Delta_2$-condition, then $L^\Psi(\Omega;\mathbb R^N)=K^\Psi(\Omega;\mathbb R^N),$
and this space becomes Banach when equipped with the Luxemburg norm
\begin{equation}
\|v\|_{L^\Psi(\Omega;\mathbb R^N)}
:=\inf\left\{\sigma>0:\int_\Omega\Psi\!\left(x,\frac{|v|}{\sigma}\right)\,dx\le1\right\}.
\end{equation}

The Musielak-Orlicz-Sobolev space is defined by
\begin{equation}
W^{1,\Psi}(\Omega;\mathbb R^N)
:=\left\{v\in L^\Psi(\Omega;\mathbb R^N):
Dv\in L^\Psi(\Omega;\mathbb R^{N\times n})\right\},
\end{equation}
and is endowed with the norm
\begin{equation}
\|v\|_{W^{1,\Psi}(\Omega;\mathbb R^N)}
=
\|v\|_{L^\Psi(\Omega;\mathbb R^N)}
+
\|Dv\|_{L^\Psi(\Omega;\mathbb R^{N\times n})}.
\end{equation}

The space $W^{1,\Psi}_0(\Omega;\mathbb{R}^N)$ is defined as the closure of
$C^\infty_0(\Omega;\mathbb{R}^N)$ with respect to the
$W^{1,\Psi}(\Omega;\mathbb{R}^N)$ norm. Throughout the paper, when
$N=1$, we use the abbreviations
$$L^\Psi(\Omega):=L^\Psi(\Omega;\mathbb{R}),
\qquad
W^{1,\Psi}(\Omega):=W^{1,\Psi}(\Omega;\mathbb{R}).$$
Further details on Musielak-Orlicz and Musielak-Orlicz-Sobolev spaces can be found, for example, in \cite{harjulehto2019generalized,Benkirane_Musielak_2014,Diening_2005,Fan_2012,Fan_2010,Harjulehto_2016,Musielak_1983} and references therein.

We conclude this subsection with a standard iteration lemma from \cite{2003_Giusti_Direct_methods_in_the_calculus_of_variations}, which will be used in the proof of the main result.
\begin{lemma}
Let $h:[\sigma_1,\sigma_2]\to[0,\infty)$
be a bounded function. Suppose there exist constants $\vartheta\in(0,1)$, $A,B\ge0$ and $m_1,m_2\ge0$ such that
\begin{equation}
h(t)\le\vartheta h(s)+
\frac{A}{(s-t)^{m_1}}+\frac{B}{(s-t)^{m_2}}
\end{equation}
for every $\sigma_1\le t<s\le\sigma_2$. Then there exists a constant $c=c(\vartheta,{m_1},{m_2})>0$ such that
\begin{equation}
h(\sigma_1)\le\frac{cA}{(\sigma_2-\sigma_1)^{m_1}}+\frac{cB}{(\sigma_2-\sigma_1)^{m_2}}.
\end{equation}
\end{lemma}
\section{Lavrentiev phenomenon and Sobolev-Poincar\'e inequality}\label{sec_4}
\noi In this preparatory section, we establish two technical results required for the proof of the main result, that is, Theorem~\ref{main_thm}. We first address the absence of the Lavrentiev phenomenon and then conclude the section by establishing a Sobolev-Poincar\'e type inequality.
\begin{theorem}
\label{thm:Lavrentiev}
Assume that \eqref{coeff_1}--\eqref{condition_2} hold, let $B_R\Subset\widetilde B\Subset\Omega$, and let $f\in W^{1,\Psi}_{\mathrm{loc}}(\Omega)$ satisfy
\[
\int_{\widetilde B}\Psi(x,Df)\,dx<\infty.
\]
Then there exists a sequence $\{f_k\}\subset W^{1,\infty}(B_R)$
such that
\begin{equation}\label{eq:thm_4.1}
f_k\longrightarrow f
\quad\text{strongly in }W^{1,G}(B_R)\quad \text{ and }\quad
\int_{B_R}\Psi(x,Df_k)\,dx
\longrightarrow
\int_{B_R}\Psi(x,Df)\,dx .
\end{equation}
\end{theorem}

\begin{proof}
The argument follows the strategy of \cite[Theorem~3.1]{Byun2020}, adapted to the present two-coefficient setting. Fix $0<\ve_0<1$
so that $B\equiv B_R\Subset B_{R+\ve_0}\Subset\wdt B\Subset\om$. Let $\rho\in C_0^\infty(B_1)$
be a non-negative standard mollifier satisfying
$\int_{\mathbb R^n}\rho\,dx=1.$
Now, we define
$\rho_\varepsilon(x):=\frac1{\varepsilon^n}\rho\!\left(\frac{x}{\varepsilon}\right)$ for $x\in B_\ve$ with $0<\ve<\ve_0$ such that $\int_{\mathbb R^n}\rho_\ve\,dx=1,\, 0\le\rho_\varepsilon\le c(n)\ve^{-n}$ and $|D\rho_\varepsilon|\le c(n)\ve^{-(n+1)}$.
The mollification of $f$ is given by
$f_\varepsilon(x):=(f*\rho_\varepsilon)(x),$
which belongs to
$C^\infty(\widetilde B_\varepsilon).$ Define
\[a_\varepsilon(x):=\inf_{y\in B_\varepsilon(x)}a(y), \qquad b_\varepsilon(x):=\inf_{y\in B_\varepsilon(x)}b(y),\]
and
\[\Psi_\varepsilon(x,z):=a_\varepsilon(x)G(|z|)+b_\varepsilon(x)H(|z|),
\]
together with
\[\Psi_{\varepsilon,1}(x,z):=\Psi_\varepsilon(x,z)+1.\]
By Jensen's inequality, we have
\begin{align}\label{eq:4}
G(|Df_\varepsilon(x)|)&=G\bigl(|(Df*\rho_\varepsilon)(x)|\bigr)\nonum
\\
&\le\int_{\mathbb R^n}G(|Df(x-y)|)\rho_\varepsilon(y)\,dy\le c\,\varepsilon^{-n}.
\end{align}
Moreover, $\Psi_\varepsilon(x,z)\le\Psi(x,z),$
and by the Hölder continuity of the coefficients,
\[
a(x)-a_\varepsilon(x)\le[a]_{0,\alpha}\varepsilon^\alpha,\qquad b(x)-b_\varepsilon(x)\le[b]_{0,\beta}\varepsilon^\beta.
\]

We claim that there exists a constant \(c>0\), independent of \(\varepsilon\), such that
\begin{equation}\label{eq:5}
\Psi(x,z)\le c\,\Psi_{\varepsilon,1}(x,z)
\end{equation}
for every \(x\in B_{R}\) and every \(z\in\mathbb R^n\) satisfying $G(|z|)\le c\,\varepsilon^{-n}.$ To this end, assume in addition that $\ve_0$ is so small that $[a]_{0,\alpha}\ve_0^\alpha+[b]_{0,\beta}\ve_0^\beta\le\mu/2$. Then $a_\ve(x)+b_\ve(x)\ge\mu/2$ for every $x\in B_R$, so that, arguing as in the proof of Lemma~\ref{lem:modular_equivalence} with $\mu$ replaced by $\mu/2$, there exists a constant $C_1'\equiv C_1'(\mu,G,H)\in(0,1]$ such that
\[
C_1'\,G(|z|)\le\Psi_{\varepsilon,1}(x,z)\qquad\text{for every }x\in B_R\text{ and }z\in\mathbb R^n.
\]
Now, for any \(\delta\in(0,1)\),
$$\begin{aligned}
\Psi_{\varepsilon,1}(x,z)
&= \delta\Psi_1(x,z)+\delta(a_\varepsilon(x)-a(x))G(|z|)\\
&\quad+\delta(b_\varepsilon(x)-b(x))H(|z|)
+(1-\delta)\Psi_{\varepsilon,1}(x,z)\\
&\ge \delta\Psi(x,z)-\delta[a]_{0,\alpha}\varepsilon^\alpha G(|z|)
-\delta[b]_{0,\beta}\varepsilon^\beta H(|z|)+(1-\delta)C_1'G(|z|).
\end{aligned}$$
Using the growth condition \eqref{condition_2} together with \eqref{eq:4} and observing that
$$\varepsilon^\beta\left(1+\varepsilon^{-\gamma}\right)
=\varepsilon^\beta+\varepsilon^{\beta-\gamma}
\le 2,$$
where $\gamma=\min\{\alpha,\beta\}$, we obtain, for some constant $c_2\ge1$ independent of $\varepsilon$ and $\delta$,
\begin{align}\label{eq:6}
\Psi_{\varepsilon,1}(x,z)
&\ge\delta\Psi(x,z)-\delta[a]_{0,\alpha}\varepsilon^\alpha G(|z|)\nonum\\
&\qquad-c_2\,\delta[b]_{0,\beta}
\varepsilon^\beta
\left(1+\varepsilon^{-\gamma}\right)G(|z|)+(1-\delta)C_1'G(|z|)\nonum\\
&\ge\delta\Psi(x,z)+\Bigl((1-\delta)C_1'-\delta[a]_{0,\alpha}-2c_2\,\delta[b]_{0,\beta}\Bigr)G(|z|).
\end{align}
Choosing $\delta:=\frac{C_1'}{1+C_1'+[a]_{0,\alpha}+2c_2[b]_{0,\beta}}$ in \eqref{eq:6}, so that $(1-\delta)C_1'-\delta[a]_{0,\alpha}-2c_2\delta[b]_{0,\beta}\ge0$, we infer that \eqref{eq:5} holds with $c=1/\delta$. Consequently, applying Jensen's inequality, we get
\[
\begin{aligned}
\Psi(x,Df_\varepsilon(x))&\le c\,\Psi_{\varepsilon,1}(x,Df_\varepsilon(x))\\
&\le c\int_{B_\varepsilon(x)}\Psi_\varepsilon(y,Df(y))\rho_\varepsilon(x-y)\,dy+c\\
&\le c\int_{B_\varepsilon(x)}\Psi(y,Df(y))\rho_\varepsilon(x-y)\,dy+c\\
&=c\,[\Psi(\cdot,Df)*\rho_\varepsilon](x)+c.
\end{aligned}
\] 
Since $\Psi(\cdot,Df(\cdot))*\rho_\varepsilon
\longrightarrow
\Psi(\cdot,Df(\cdot))\text{ strongly in }L^1(B_R),$
we may apply a variant of the Lebesgue dominated convergence theorem to extract a sequence of functions $$\{f_k\}:=\{f_{\varepsilon_k}\}\subset C^\infty(B),
\quad \varepsilon_k\searrow0,$$
such that \eqref{eq:thm_4.1} holds.
\end{proof}

We next establish a Sobolev-Poincar\'e type inequality in the Musielak-Orlicz-Sobolev space $\W(\om)$. To this end, we first recall the corresponding Sobolev-Poincar\'e inequality for the Orlicz-Sobolev space $W^{1,\Phi}(\om)$, which will be used in the proof. We refer to \cite[Theorem~7]{Diening_2008} for further details.

\begin{lemma}\label{Sobolev_Poincare_Young}
Let $\Phi\in\mathcal{N}$ with $s(\Phi)\ge1$. Then there exists $\theta=\theta\bigl(n,s(\Phi)\bigr)\in(0,1)$
such that, for every $f\in W^{1,\Phi}(B_R)$, we have
\begin{equation*}
\dashint_{B_R}
\Phi\!\left(
\left|
\frac{f-(f)_{B_R}}{R}
\right|
\right)\,dx
\le c\left(\dashint_{B_R}\bigl[\Phi(|Df|)\bigr]^\theta\,dx \right)^{\frac1\theta}
\end{equation*}
for some constant $c\equiv c\bigl(n,s(\Phi)\bigr).$
\end{lemma}
We now prove the Sobolev-Poincar\'e type inequality announced at the beginning of this section.
\begin{theorem}\label{Sobolev-Poincare}
Let $\Psi:\Omega\times[0,\infty)\to[0,\infty)$ be defined by \eqref{function_psi} and suppose that the assumptions \eqref{coeff_1}--\eqref{condition_2} are satisfied. Then there exists
$$\theta=\theta\bigl(n,s(G),s(H),\gamma\bigr)\in(0,1)$$
such that, for every $v\in W^{1,\Psi}(B_r)$ and every ball $B_r\subset\Omega$ with $r\le1$, the estimate
\begin{equation*}
\dashint_{B_r}
\Psi_1\!\left(x,\frac{v-(v)_{B_r}}{r}\right)\,dx\le c\left[1+\left(\int_{B_r}G(|Dv|)\,dx\right)^{\frac{\gamma}{n}}\right]
\left[\dashint_{B_r}\bigl(\Psi_1(x,Dv)\bigr)^\theta\,dx\right]^{\frac1\theta}
\end{equation*}
holds, where $c\equiv c\bigl(n,\kappa,s(G),s(H),\alpha,\beta,\mu, [a]_{0,\alpha},[b]_{0,\beta}\bigr).$
\end{theorem}

\begin{proof}
Using the definition of $\Psi_1$, we decompose the left-hand side as
$$\dashint_{B_r} \Psi_1\!\left(x,\frac{|v-(v)_{B_r}|}{r}\right)dx = I_a + I_b + 1,$$
where
$$I_a := \dashint_{B_r} a(x)\,G\!\left(\frac{|v-(v)_{B_r}|}{r}\right)dx, \qquad I_b := \dashint_{B_r} b(x)\,H\!\left(\frac{|v-(v)_{B_r}|}{r}\right)dx.$$
 Since $\Psi_1(x,t) \geq 1$ for all $x$ and $t$, we have the trivial bound $1 \leq \bigl(\dashint_{B_r}[\Psi_1(x,|Dv|)]^\theta\bigr)^{1/\theta}$ for any $\theta \in (0,1)$. It remains to estimate $I_a$ and $I_b$.
 
\medskip
\noi {\bf Step 1:} The $G^{1+\gamma/n}$-Poincar\'e Estimate. Following the approach of \cite[Theorem 4.2, Step 1]{Baasandorj2020}, since $s(G) \geq 1$, there exists
$$\theta_1 \equiv \theta_1(n,s(G),\gamma) \in \left(\left(1-\tfrac{1}{n}\right)\!\left(1+\tfrac{\gamma}{n}\right),\,1\right)$$ such that
\begin{equation}\label{eq:9}
\dashint_{B_r} G^{1+\gamma/n}\!\!\left(\frac{|v-(v)_{B_r}|}{r}\right)dx \;\leq\; c_1 \left(\dashint_{B_r} G^{\theta_1}(|Dv|)\,dx\right)^{\!(1+\gamma/n)/\theta_1}
\end{equation}
for $c_1 \equiv c_1(n,s(G),\gamma)$. In particular, applying H\"older inequality, we obtain
\begin{equation*}
\dashint_{B_r} G\left(\frac{|v-(v)_{B_r}|}{r}\right)dx \;\leq\; c_1 \left(\dashint_{B_r} G^{\theta_1}(|Dv|)\,dx\right)^{\!1/\theta_1}
\end{equation*}

\medskip
\noi{\bf Step 2:} Estimate of $I_a$. We consider two complementary sub-cases.
 
\noi{\textit{Case} (a):} $\sup_{B_r} a(\cdot) \leq 4[a]_{0,\alpha}r^\alpha$. For every $x \in B_r$, we get
$$I_a \;\leq\; 4[a]_{0,\alpha}r^\alpha \dashint_{B_r} G\!\left(\frac{|v-(v)_{B_r}|}{r}\right)dx.$$
Applying Lemma \ref{Sobolev_Poincare_Young} to $G$ with exponent $\vartheta_G \equiv \vartheta_G(n,s(G)) \in (0,1)$
$$\dashint_{B_r} G\!\left(\frac{|v-(v)_{B_r}|}{r}\right)dx \;\leq\; c\left(\dashint_{B_r} [G(|Dv|)]^{\vartheta_G}\,dx\right)^{\!1/\vartheta_G}.$$
Therefore, by Lemma \ref{lem:modular_equivalence} (lower bound $C_1 G(t) \leq \Psi_1(x,t)$), we have
$$I_a \;\leq\; c\,\left(\dashint_{B_r} [\Psi_1(x,|Dv|)]^{\vartheta_G}\,dx\right)^{\!1/\vartheta_G}.$$
 
\noi{\textit{Case} (b):} $\sup_{B_r} a(\cdot) > 4[a]_{0,\alpha}r^\alpha$. There exists $y \in B_r$ such that $a(y) > 4[a]_{0,\alpha}r^\alpha$. For every $x \in B_r$,
$$|a(x) - a(y)| \leq 2[a]_{0,\alpha}r^\alpha,$$
so $\frac{1}{2}a(y) \leq a(x) \leq \frac{3}{2}a(y)$ for all $x \in B_r$. In particular,
$$I_a \leq \dashint_{B_r} a(x)\,G\!\left(\frac{|v-(v)_{B_r}|}{r}\right)dx \leq \frac{3}{2}\,a(y) \dashint_{B_r} G\!\left(\frac{|v-(v)_{B_r}|}{r}\right)dx.$$
By Lemma \ref{lem:operations-N} the function $\Phi(t) := a(y)G(t)$ belongs to $\mathcal{N}$ with $s(\Phi) = s(G)$. Applying Lemma \ref{Sobolev_Poincare_Young} to $\Phi$ with exponent $\vartheta_0 \equiv \vartheta_0(n,s(G)) \in (0,1)$ yields
$$a(y)\dashint_{B_r} G\!\left(\frac{|v-(v)_{B_r}|}{r}\right)dx \leq c\left(\dashint_{B_r}[a(y)G(|Dv|)]^{\vartheta_0}\,dx\right)^{\!1/\vartheta_0}.$$
Since $a(y)G(t) \leq \Psi_1(x,t)$ for all $x \in B_r$, we get
$$I_a \;\leq\; c\left(\dashint_{B_r}[\Psi_1(x,|Dv|)]^{\vartheta_0}\,dx\right)^{\!1/\vartheta_0}.$$

\noi {\bf Step 3:} Estimate of $I_b$. We again split into two complementary sub-cases.

\noi {\textit{Case} $(c)$:}  $\sup_{B_r} b(\cdot) > 4[b]_{0,\beta}r^\beta$. For some $\vartheta_H\equiv\vartheta_H(n,s(H))\in(0,1)$, repeating the argument from Step~2, Case~(b), with $H$ in place of $G$, we obtain
\[
I_b \leq c\left(\dashint_{B_r}[\Psi_1(x,|Dv|)]^{\vartheta_H}\,dx\right)^{1/\vartheta_H}.
\]

\noi {\textit{Case} $(d)$:} $\sup_{B_r} b(\cdot) \leq 4[b]_{0,\beta}r^\beta$. Since $b(x) \leq 4[b]_{0,\beta}r^\beta$ for every $x \in B_r$, applying the growth condition $H(t) \leq \kappa(G(t)+G^{1+\gamma/n}(t))$, we get
\begin{align*}
    I_b &\leq 4[b]_{0,\beta}r^\beta \dashint_{B_r} H\!\left(\frac{|v-(v)_{B_r}|}{r}\right)dx\\
    &\leq 4\kappa[b]_{0,\beta}r^\beta\left[\dashint_{B_r} G\!\left(\frac{|v-(v)_{B_r}|}{r}\right)dx + \dashint_{B_r} G^{1+\gamma/n}\!\!\left(\frac{|v-(v)_{B_r}|}{r}\right)dx\right].
\end{align*}

We bound each term on the right-hand side separately.
 
First bracket term. By Lemma \ref{Sobolev_Poincare_Young} applied to $G$ with exponent $\theta_1$ (where $\theta_1$ is as in Step 1) implies that
$$r^\beta\dashint_{B_r} G\!\left(\frac{|v-(v)_{B_r}|}{r}\right)dx \;\leq\; cr^\beta\left(\dashint_{B_r} G^{\theta_1}(|Dv|)\,dx\right)^{\!1/\theta_1}.$$
Since $r \leq 1$ implies $r^\beta \leq 1$. Using Lemma \ref{lem:modular_equivalence}, we get
\begin{equation}\label{eq:7}
    r^\beta\dashint_{B_r} G\!\left(\frac{|v-(v)_{B_r}|}{r}\right)dx \;\leq\; c\left(\dashint_{B_r} [\Psi_1(x,|Dv|)]^{\theta_1}\,dx\right)^{\!1/\theta_1}.
\end{equation}
 
Second bracket term. Applying \eqref{eq:9} and then H\"older inequality gives
\begin{align*}
    &r^{\beta}\dashint_{B_r} G^{1+\gamma/n}\!\!\left(\frac{|v-(v)_{B_r}|}{r}\right)dx \leq c_1 r^\beta\left(\dashint_{B_r} G^{\theta_1}(|Dv|)\,dx\right)^{\!(1+\gamma/n)/\theta_1}\\
    &\quad\leq c_1 r^{\beta} \left(\dashint_{B_r} G^{\theta_1}(|Dv|)\,dx\right)^{\!\gamma/(\theta_1 n)}\!\left(\dashint_{B_r} G^{\theta_1}(|Dv|)\,dx\right)^{\!1/\theta_1}\\
    &\quad\leq c_1 r^{\beta-\gamma} \left(\int_{B_r} G(|Dv|)\,dx\right)^{\!\gamma/n}\!\left(\dashint_{B_r} G^{\theta_1}(|Dv|)\,dx\right)^{\!1/\theta_1}
\end{align*}
Using $r^{\beta-\gamma} \leq 1$ and Lemma \ref{lem:modular_equivalence}, we get
\begin{multline}\label{eq:8}
    r^\beta\dashint_{B_r} G^{1+\gamma/n}\!\!\left(\frac{|v-(v)_{B_r}|}{r}\right)dx \\ \leq c\left(\dashint_{B_r} G(|Dv|)\,dx\right)^{\!\gamma/n}\!\left(\dashint_{B_r} [\Psi_1(x,|Dv|)]^{\theta_1}\,dx\right)^{\!1/\theta_1}.
\end{multline}
 
Combining \eqref{eq:7} and \eqref{eq:8}, we conclude
$$I_b \;\leq\; c\left[1 + \left(\int_{B_r} G(|Dv|)\,dx\right)^{\!\gamma/n}\right]\left(\dashint_{B_r} [\Psi_1(x,|Dv|)]^{\theta_1}\,dx\right)^{\!1/\theta_1}. $$

 \medskip
\noi{\bf Step 4:} Conclusion. Setting $\theta:=\max\{\vartheta_G,\vartheta_0,\vartheta_H,\theta_1\}\in(0,1)$
and collecting the constants
\[
c\equiv c\bigl(n,\kappa,s(G),s(H),\alpha,\beta,\mu,[a]_{0,\alpha},[b]_{0,\beta}\bigr)
\]
from Steps~1--3, we obtain
\begin{multline*}
\dashint_{B_r}
\Psi_1\!\left(x,\frac{|v-(v)_{B_r}|}{r}\right)\,dx\\
\le c\left[1+\left(\int_{B_r}G(|Dv|)\,dx\right)^{\gamma/n}\right]
\left(\dashint_{B_r}[\Psi_1(x,|Dv|)]^{\theta}\,dx
\right)^{1/\theta}.
\end{multline*}
This completes the proof of the Theorem~\ref{Sobolev-Poincare}.
\end{proof}

\section{Fractional differentiability result}\label{sect_5}
\noi We begin this section with the following lemma, which may be viewed as a higher integrability version of Theorem~\ref{thm:Lavrentiev}.
\begin{lemma}\label{lem:approx-higher}
Fix $\delta\ge0$ and let $v\in W^{1,\Psi}_{\rm loc}(\Omega)$ satisfy $[\Psi(\cdot,Dv)]^{1+\delta}\in L^1_{\rm loc}(\Omega)$. Then there exists a sequence $\{v_m\}\subset W^{1,\infty}_{\rm loc}(\Omega)$ such that $v_m\to v$ strongly in $W^{1,G^{1+\delta}}_{\rm loc}(\Omega)$ and
\[
\int_B[\Psi(x,Dv_m)]^{1+\delta}\,dx\longrightarrow\int_B[\Psi(x,Dv)]^{1+\delta}\,dx
\]
for every ball $B\Subset\Omega$.
\end{lemma}

\begin{proof}
Set $\Psi_\delta(x,z):=a(x)^{1+\delta}G^{1+\delta}(|z|)+b(x)^{1+\delta}H^{1+\delta}(|z|)$. Since $t\mapsto t^{1+\delta}$ is convex and increasing, $(A+B)^{1+\delta}\approx A^{1+\delta}+B^{1+\delta}$ for $A,B\ge0$, so
\[
\Psi_\delta(x,z)\approx[\Psi(x,z)]^{1+\delta}
\]
uniformly on $\Omega\times\mathbb R^n$. It thus suffices to prove the same convergences with $\Psi_\delta$ in place of $\Psi^{1+\delta}$.

Since $a\in C^{0,\alpha}(\Omega)$ and $t\mapsto t^{1+\delta}$ is locally Lipschitz, $a^{1+\delta}\in C^{0,\alpha}(\Omega)$, and likewise $b^{1+\delta}\in C^{0,\beta}(\Omega)$. Also from Lemma \ref{lem:operations-N}, we can see that $G^{1+\delta},H^{1+\delta}\in\mathcal N$. Hence, applying Theorem \ref{thm:Lavrentiev} to $v$ gives $\{v_m\}\subset W^{1,\infty}_{\rm loc}(\Omega)$ with
\[
v_m\to v\ \text{in }W^{1,G^{1+\delta}}_{\rm loc}(\Omega),\qquad \int_B\Psi_\delta(x,Dv_m)\,dx\to\int_B\Psi_\delta(x,Dv)\,dx
\]
for every $B\Subset\Omega$. The claim follows from the equivalence $\Psi_\delta\approx\Psi^{1+\delta}$.
\end{proof}
The following auxiliary lemma will be used to prove the existence and uniqueness result.
\begin{lemma}\label{lem:existence}
Assume that \eqref{coeff_1}--\eqref{condition_2} hold. Let $B\Subset\Omega$ be a ball and let
$S:B\to\mathbb{R}^n$ be a measurable vector field such that
$\Psi(\cdot,S)\in L^1(B)$ and
\begin{equation*}
    \operatorname{div}T(x,S)=0\quad\text{in }B
\end{equation*}
in the sense of distributions, where
$T:B\times\mathbb{R}^n\to\mathbb{R}^n$ satisfies the growth condition
\begin{equation}
|T(x,z)|\le L\left(a(x)\frac{G(|z|)}{|z|}+b(x)\frac{H(|z|)}{|z|}\right)
\end{equation}
for every $x\in B$ and $z\in\mathbb{R}^n$, with $L\ge1$.

Then
\begin{equation}
\int_B \langle T(x,S),D\varphi\rangle\,dx=0
\end{equation}
for any $\varphi\in W_0^{1,1}(B)$ satisfying
$\Psi(\cdot,D\varphi)\in L^1(B)$.
\end{lemma}

\begin{proof}
    By a standard scaling and translation argument, it is sufficient to consider the case
$B\equiv B_1(0)$. Extending $\varphi$ by zero outside $B$, we may assume that
$\varphi\in W_0^{1,\Psi}(\mathbb{R}^n)$. Moreover, by McShane's extension theorem \cite[Corollary 1]{McShane},
the coefficients $a$ and $b$ can be extended to $\mathbb{R}^n$ while preserving
their H\"older continuity with the same exponent $\alpha$ and $\beta$ respectively. Let $\rho\in C_0^\infty(B_1(0))$ be a non-negative mollifier satisfying
$\int_{\mathbb{R}^n}\rho\,dx=1,$
and define
$$\rho_\varepsilon(x):=\frac{1}{\varepsilon^n}
\rho\!\left(\frac{x}{\varepsilon}\right),
\qquad x\in B_\varepsilon(0).
$$
Then
$$\rho_\varepsilon\in C_0^\infty(B_\varepsilon),
\qquad
\int_{\mathbb{R}^n}\rho_\varepsilon\,dx=1,
\qquad
0\le\rho_\varepsilon\le c(n)\varepsilon^{-n}.$$

For $0<\varepsilon<\frac1{100}$, we define $\tilde{\varphi}(x):=\varphi\!\left(\frac{x}{1-2\varepsilon}\right)$ and observe that $\tilde{\varphi}$ vanishes outside the ball $B_{1-2\ve}$. Similarly, we set
$$\tilde{a}(x):=a\!\left(\frac{x}{1-2\varepsilon}\right),
\qquad
\tilde{b}(x):=b\!\left(\frac{x}{1-2\varepsilon}\right).$$
We consider the mollification $\varphi_\varepsilon(x):=(\tilde{\varphi}*\rho_\varepsilon)(x)\in C^\infty_0(B_{1-\ve})$ and introduce the auxilary functions 
$$a_\varepsilon(x):=\inf_{y\in B_\varepsilon(x)}\tilde{a}(y),
\qquad
b_\varepsilon(x):=\inf_{y\in B_\varepsilon(x)}\tilde{b}(y).$$
and
$$
\Psi_\varepsilon(x,z):=a_\varepsilon(x)G(|z|)+b_\varepsilon(x)H(|z|), \qquad \Psi_{\ve,1}(x,z):=\Psi_\varepsilon(x,z)+1.$$

By Jensen's inequality, we get
\begin{equation}\label{eq:13}
G(|D\varphi_\varepsilon(x)|)=G\!\left(|(D\tilde{\varphi}*\rho_\varepsilon)(x)|\right)\le\int_{\mathbb{R}^n}G(|D\tilde{\varphi}(x-y)|)
\rho_\varepsilon(y)\,dy\le c\,\varepsilon^{-n}.
\end{equation}
Then, for every $x\in B_1$, the definition of convolution and convexity of $G$ and $H$ yields
\begin{align}\label{eq:11}
\Psi_{\varepsilon,1}(x,D\varphi_\varepsilon(x))&\le \int_{B_\varepsilon(x)}\Psi_\varepsilon(y,D\tilde\varphi(y))\rho_\varepsilon(x-y)\,dy+1\nonum\\
&\le \frac{1}{(1-2\ve)^s}\int_{B_\varepsilon(x)}\Psi_\ve\!\left(x,D\varphi\left(\frac{y}{1-2\ve}\right)\right)\rho_\varepsilon(x-y)\,dy+c\nonum\\
&\le 2^s\int_{B_\varepsilon(x)}\Psi\!\left(\frac{y}{1-2\ve},D\varphi\left(\frac{y}{1-2\ve}\right)\right)\rho_\varepsilon(x-y)\,dy+c\nonum\\
&\le c\,\left[\Psi\left(\frac{\cdot}{1-2\ve},D\varphi\left(\frac{\cdot}{1-2\ve}\right)\right)*\rho_\varepsilon\right](x)+c,
\end{align}
where $s:=\max\{s(G)+1,s(H)
+1\}.$ Now using \eqref{eq:13}, we prove the estimate
\begin{align}\label{eq:12}
\Psi(x,D\varphi_\varepsilon(x))&\le |a(x)-a_\ve(x)|G(|D\varphi_\varepsilon(x)|)+|b(x)-b_\ve(x)|H(|D\varphi_\varepsilon(x)|)\nonum\\
&\qquad+\Psi_\ve(x,D\varphi_\varepsilon(x))\nonum\\
&\le cG(|D\varphi_\varepsilon(x)|)+c \kappa[b]_\beta \ve^\beta \left(1+G^{\gamma/n}(|D\varphi_\varepsilon(x)|)\right)G(|D\varphi_\varepsilon(x)|)\nonum\\
&\qquad+\Psi_{\ve,1}(x,D\varphi_\varepsilon(x))\nonum\\
&\le cG(|D\varphi_\varepsilon(x)|)+c\left(\ve^\beta+c^{\gamma/n}\ve^{\beta-\gamma}\right)G(|D\varphi_\varepsilon(x)|)+\Psi_{\ve,1}(x,D\varphi_\varepsilon(x))\nonum\\
&\le c\,\Psi_{\ve,1}(x,D\varphi_\varepsilon(x)).
\end{align}
In the last inequality, we used Lemma~\ref{lem:modular_equivalence} together with the fact that $\beta-\gamma$ is non-negative. From \eqref{eq:11} and \eqref{eq:12}, we conclude that
\begin{align}
\Psi(x,D\varphi_\varepsilon(x))\le c\,\left[\Psi\left(\frac{\cdot}{1-2\ve},D\varphi\left(\frac{\cdot}{1-2\ve}\right)\right)*\rho_\varepsilon\right](x)+c.
\end{align}
By the assumptions, we know that $\Psi\!\left(\frac{\cdot}{1-2\varepsilon},
D\varphi\!\left(\frac{\cdot}{1-2\varepsilon}\right)\right)
\in L^1(\mathbb{R}^n)$
and
$$
\left[
\Psi\!\left(
\frac{\cdot}{1-2\varepsilon},
D\varphi\!\left(\frac{\cdot}{1-2\varepsilon}\right)
\right)
*\rho_\varepsilon
\right](x)
\longrightarrow
\Psi(\cdot,D\varphi)
\quad\text{strongly in }L^1(B)
$$
as $\varepsilon\to0$. Hence, by the standard mollification argument, there exists a sequence
$\{\varphi_k\}\subset C_0^\infty(B)$
such that
\[\Psi(\cdot,D\varphi_k)\longrightarrow\Psi(\cdot,D\varphi)\quad\text{strongly in }L^1(B).
\]
Moreover, an application of Lemma \ref{lem:Young epsilon} yields
$$|\langle T(x,S),D\varphi_k\rangle|
\le
c\bigl[\Psi(x,S)+\Psi(x,D\varphi_k)\bigr].$$
Therefore, by the dominated convergence theorem,
\[\langle T(x,S),D\varphi_k\rangle\longrightarrow
\langle T(x,S),D\varphi\rangle\quad\text{strongly in }L^1(B).
\]
Passing to the limit in
\[\int_B\langle T(x,S),D\varphi_k\rangle\,dx=0\]
completes the proof.
\end{proof}

In the remainder of this section, we fix a ball $B:=B_R(x_0)\Subset\Omega,\,R\le 1$
and study the existence and regularity estimates for weak solutions to the Dirichlet problem
\begin{equation}\label{eq:dirichlet}
\begin{cases}
-\operatorname{div}\A(x,Dw)=0 & \text{in } B,\\[2mm]
w-w_0\in W_0^{1,\Psi}(B),
\end{cases}
\end{equation}
where the boundary datum
$w_0\in W^{1,\Psi}(B)$
satisfies
\begin{equation}\label{eq:boundary1}
\Psi(x,Dw_0)\in L^{1+\delta}(B)\qquad\text{for some }\delta>0,
\end{equation}
and
\begin{equation}\label{eq:boundary2}
\|\Psi(\cdot,Dw_0)\|_{L^1(B)}\le L_0
\end{equation}
for a prescribed constant $L_0\ge0$.

Throughout this section, the vector field $\A(\cdot,\cdot)$ is assumed to satisfy
\eqref{growth_conditions}. For each $m\in\mathbb N$, we introduce the regularized vector field
\begin{equation}\label{eq:regularized-vector-field}
\A_m(x,z):=\A(x,z)+\varepsilon_m K'(|z|)\frac{z}{|z|},
\end{equation}
where $K\in\mathcal N$ satisfies $s(K)\ge1$ and $\limsup_{t\to\infty}\frac{H(t)}{K(t)}<\infty,$
while $\{\varepsilon_m\}_{m=1}^\infty$ is a sequence of positive numbers with
$\varepsilon_m\searrow0$.

Furthermore, let $\{w_m\}_{m=1}^\infty\subset W^{1,\infty}(B)$ be a sequence satisfying
\begin{equation*}
\int_B\Psi(x,Dw_m)\,dx
\longrightarrow
\int_B\Psi(x,Dw_0)\,dx,
\end{equation*}
and
\begin{equation*}
\int_B[\Psi(x,Dw_m)]^{1+\delta}\,dx
\longrightarrow
\int_B[\Psi(x,Dw_0)]^{1+\delta}\,dx.
\end{equation*}
The existence of such a sequence follows from Lemma~\ref{lem:approx-higher}.

For every $m\in\mathbb N$, let us consider another sequence $\{v_m\}_{m=1}^\infty$ denote the unique weak solution of
\begin{equation}\label{eq:regularized-dirichlet}
\begin{cases}
-\operatorname{div}\A_m(x,Dv_m)=0 & \text{in } B,\\[2mm]
v_m-w_m\in W_0^{1,K}(B).
\end{cases}
\end{equation}

By the standard regularity theory (see \cite{Esposito_2006}), it follows that
\begin{equation*}v_m\in W_{\mathrm{loc}}^{1,\infty}(B).
\end{equation*}

We first prove the following two higher integrability lemmas, originally established for double phase functionals in \cite{Mingione_2019}. These lemmas are used to establish the fractional differentiability result, a key step in the proof of the main theorem.

\begin{lemma}\label{lem:higher_int_local}
Assume that \eqref{coeff_1}--\eqref{condition_2} and \eqref{growth_conditions} are satisfied. Let the sequence $\{v_m\}_{m=1}^\infty\subset W^{1,\infty}_{\mathrm{loc}}(B)$ be the sequence of weak solutions to the perturbed problem \eqref{eq:regularized-dirichlet} and suppose that
\begin{equation}\label{eq:L1-bound}
\sup_{m\in\mathbb N}
\int_B\Psi(x,Dv_m)\,dx
\le L_1
\end{equation}
for some constant $L_1>0$. Then there exist constants
\[\delta_1\equiv\delta_1\bigl(n,s(G),s(H),s(K),\alpha,\beta,[a]_{0,\alpha},[b]_{0,\beta},L_1\bigr)>0
\]
and
\[c\equiv c\bigl(n,\kappa,s(G),s(H),s(K),\alpha,\beta,[a]_{0,\alpha},
[b]_{0,\beta},L_1\bigr)>0
\] such that
\begin{align}\label{eq:higher-int-local}
&\left(
\dashint_{B_\rho}
\bigl[\Psi(x,Dv_m)+\varepsilon_mK(|Dv_m|)\bigr]^{1+\delta_1}
\,dx
\right)^{\frac1{1+\delta_1}}\nonum\\
&\qquad\quad\le c\dashint_{B_{2\rho}}
\bigl[\Psi(x,Dv_m)+\varepsilon_mK(|Dv_m|)\bigr]
\,dx+c
\end{align}
holds for every ball $B_{2\rho}\subset B$ and every $m\in\mathbb N$.
\end{lemma}
\begin{proof}
Let $B_{2\rho}\subset B$ be arbitrary and choose a cut-off function
$\eta\in C_0^1(B_{2\rho})$ satisfying
\[
\chi_{B_\rho}\le \eta\le \chi_{B_{2\rho}},
\qquad
|D\eta|\le \frac{4}{\rho}.
\]
Fix $s\ge\max\{s(G),s(H),s(K)\}+1,$
and use the test function $\varphi=\eta^s\bigl(v_m-(v_m)_{B_{2\rho}}\bigr)$
in the weak formulation of \eqref{eq:regularized-dirichlet}. Since
$v_m\in W^{1,K}(B)$ for all $m$, the function $\varphi$ is admissible. Exploiting the
monotonicity of the vector field $\A(x,\cdot)$ together with Lemma~\ref{lem:Young epsilon}, for every $\varepsilon\in(0,1)$, we obtain
\begin{align*}
\int_{B_{2\rho}}\eta^s&\bigl[\Psi(x,Dv_m)+\varepsilon_mK(|Dv_m|)\bigr]\,dx\nonum\\
&\le c\int_{B_{2\rho}}\eta^{s-1}\Bigl(a(x)G'(|Dv_m|)+b(x)H'(|Dv_m|)\nonumber\\
&\qquad\qquad\qquad+\varepsilon_mK'(|Dv_m|)
\Bigr)\left|\frac{v_m-(v_m)_{B_{2\rho}}}{\rho}\right|dx\nonumber\\
&\le c\int_{B_{2\rho}}a(x)\eta^{s-1}\left(
\varepsilon\eta\,G(|Dv_m|)+\frac1{(\varepsilon\eta)^{s(G)}}G\!\left(
\left|\frac{v_m-(v_m)_{B_{2\rho}}}{\rho}\right|\right)\right)dx\nonumber\\
&\quad+c\int_{B_{2\rho}}
b(x)\eta^{s-1}\left(\varepsilon\eta\,H(|Dv_m|)+
\frac1{(\varepsilon\eta)^{s(H)}}
H\!\left(\left|\frac{v_m-(v_m)_{B_{2\rho}}}{\rho}\right|\right)\right)dx\nonumber\\
&\quad+c\varepsilon_m\int_{B_{2\rho}}
\eta^{s-1}\left(\varepsilon\eta\,K(|Dv_m|)+
\frac1{(\varepsilon\eta)^{s(K)}}K\!\left(
\left|\frac{v_m-(v_m)_{B_{2\rho}}}{\rho}\right|\right)\right)dx.
\end{align*}

Choosing $\varepsilon>0$ sufficiently small, the terms involving
$\Psi(x,Dv_m)$ and $\varepsilon_mK(|Dv_m|)$ on the right-hand side can be absorbed into the left-hand side. Consequently,
\begin{align}\label{eq:RH2}
\dashint_{B_\rho}
&\bigl[\Psi(x,Dv_m)+\varepsilon_mK(|Dv_m|)\bigr]\,dx\nonum\\
&\qquad\le c
\dashint_{B_{2\rho}}\Biggl[\Psi\!\left(
x,\frac{v_m-(v_m)_{B_{2\rho}}}{\rho}\right)+\varepsilon_m
K\!\left(\left|\frac{v_m-(v_m)_{B_{2\rho}}}{\rho}
\right|\right)\Biggr]dx.
\end{align}
Next, we estimate the two terms on the right-hand side of
\eqref{eq:RH2}. Since $v_m\in W^{1,K}(B)$, Theorem~\ref{Sobolev-Poincare} applied to $\Psi$ yields
\begin{equation}\label{eq:RH3}
\dashint_{B_{2\rho}}
\Psi\!\left(x,\frac{v_m-(v_m)_{B_{2\rho}}}{\rho}\right)\,dx\le c\left(
\dashint_{B_{2\rho}}[\Psi_1(x,Dv_m)]^{d_1}\,dx
\right)^{\frac1{d_1}},
\end{equation}
where $d_1\equiv d_1\bigl(n,s(G),s(H),\alpha\bigr)\in(0,1).$

On the other hand, applying the Sobolev-Poincar\'e inequality (see Lemma~\ref{Sobolev_Poincare_Young}) for the function $K$ gives
\begin{equation}\label{eq:RH4}
\dashint_{B_{2\rho}}
K\!\left(\left|\frac{v_m-(v_m)_{B_{2\rho}}}{\rho}
\right|\right)\,dx
\le c\left(\dashint_{B_{2\rho}}[K(|Dv_m|)]^{d_2}\,dx
\right)^{\frac1{d_2}},
\end{equation}
where $d_2\equiv d_2\bigl(n,s(K)\bigr)\in(0,1).$ Set $d_0:=\max\{d_1,d_2\}.$
Since $d_0\in(0,1)$, H\"older's inequality together with
\eqref{eq:RH3}--\eqref{eq:RH4} implies
\begin{align}
&\dashint_{B_{2\rho}}\left[\Psi\!\left(
x,\frac{v_m-(v_m)_{B_{2\rho}}}{\rho}\right)+\varepsilon_m K\!\left(
\left|\frac{v_m-(v_m)_{B_{2\rho}}}{\rho}
\right|\right)\right]dx\nonumber\\
&\qquad\le c\left(\dashint_{B_{2\rho}}
\bigl[\Psi_1(x,Dv_m)+\varepsilon_mK(|Dv_m|)
\bigr]^{d_0}\,dx\right)^{\frac1{d_0}}.
\label{eq:RH5}
\end{align}

Combining \eqref{eq:RH2} with \eqref{eq:RH5}, we arrive at
\begin{align*}
\dashint_{B_{\rho}}
\bigl[\Psi(x,Dv_m)&+\varepsilon_mK(|Dv_m|)
\bigr]\,dx\nonum\\
&\le c\left(\dashint_{B_{2\rho}}
\bigl[\Psi_1(x,Dv_m)+\varepsilon_mK(|Dv_m|)\bigr]^{d_0}
\,dx\right)^{\frac1{d_0}}\nonum\\
&\le c\left(\dashint_{B_{2\rho}}
\bigl[\Psi(x,Dv_m)+\varepsilon_mK(|Dv_m|)\bigr]^{d_0}
\,dx\right)^{\frac1{d_0}}+c.
\end{align*}
Applying the Gehring lemma (cf. \cite[Theorem 6.6]{2003_Giusti_Direct_methods_in_the_calculus_of_variations}), we conclude that there exist constants
\[\delta_1\equiv\delta_1\bigl(n,s(G),s(H),s(K),
[a]_{0,\alpha},[b]_{0,\beta},L_1\bigr)>0\]
and
\[c\equiv c\bigl(n,\kappa,s(G),s(H),s(K),[a]_{0,\alpha},[b]_{0,\beta},L_1\bigr)>0\]
such that \eqref{eq:higher-int-local} holds. This completes the proof.
\end{proof}
Next, we aim to prove the global version of the previous lemma.

\begin{lemma}\label{lem:higher_int_global}
Assume that \eqref{coeff_1}--\eqref{condition_2}, \eqref{growth_conditions}, \eqref{eq:boundary2} and \eqref{eq:L1-bound} are satisfied. Suppose that
\begin{equation*}
\sup_{m\in\mathbb N}\int_B
\Psi(x,Dw_m)\,dx\le L_2,
\end{equation*}
for some constant $L_2>0$. Then there exist constants
\[\delta_2\equiv\delta_2\bigl(n,s(G),s(H),s(K),\nu,L,\alpha,
[a]_{0,\alpha},\beta,[b]_{0,\beta},L_1,L_2\bigr)\in(0,\delta_1]
\]
and
\[c\equiv c\bigl(n,\kappa,s(G),s(H),s(K),\nu,L,
\alpha,[a]_{0,\alpha},\beta,[b]_{0,\beta},
L_1,L_2\bigr)>0\]
such that
\begin{align}\label{eq:19}
&\left(\dashint_B\bigl[\Psi(x,Dv_m)+\varepsilon_mK(|Dv_m|)
\bigr]^{1+\sigma}dx\right)^{\frac1{1+\sigma}}\nonum\\
&\quad\qquad\le c\left(\dashint_B\bigl[
\Psi(x,Dw_m)+\varepsilon_mK(|Dw_m|)\bigr]^{1+\sigma}dx\right)^{\frac1{1+\sigma}}+c
\end{align}
holds for every $\sigma\in(0,\delta_2)$ and every $m\in\mathbb N$.
\end{lemma}
\begin{proof}
We first derive a global energy estimate for the approximate solutions.
Choosing $\varphi=v_m-w_m$ as a test function in the weak formulation of the regularized problem \eqref{eq:regularized-dirichlet}, we obtain
\[\int_B\langle \A_m(x,Dv_m),D(v_m-w_m)\rangle\,dx=0.\]
Using the monotonicity of $\A_m$, followed by \ref{eq:VPhi-est}, gives
\begin{equation}\label{seq_wm_vm}
    \int_B\bigl[\Psi(x,Dv_m)+\varepsilon_mK(|Dv_m|)\bigr]\,dx\le
c\int_B\bigl[\Psi(x,Dw_m)+\varepsilon_mK(|Dw_m|)\bigr]\,dx.
\end{equation}

Let
$B_{2\rho}(y)\subset\mathbb R^n$
be a ball with centre $y\in B$ satisfying $$\frac1{10}|B_{2\rho}(y)|<|B_{2\rho}(y)\setminus B|,$$
and let $\varphi=\eta^s(v_m-w_m),$
where $s\ge\max\{s(G),s(H),s(K)+1\}$
and
$\eta\in C_0^\infty(B_{2\rho}(y))$
is a cut-off function satisfying
\[\chi_{B_\rho}\le\eta
\le\chi_{B_{2\rho}},\qquad
|D\eta|\le\frac4\rho.\]
Since
$\operatorname{supp}\varphi
\subset
B\cap B_{2\rho}(y)$,
the function $\varphi$ is admissible in the weak formulation of
\eqref{eq:regularized-dirichlet}. Consequently,
\begin{align}\label{eq:18}
&\int_{B\cap B_{2\rho}(y)}\eta^s\bigl[\Psi(x,Dv_m)
+\varepsilon_mK(|Dv_m|)\bigr]\,dx\nonumber\\
&\qquad\le c\int_{B\cap B_{2\rho}(y)}
\eta^{s-1}|\A_m(x,Dv_m)|\left|\frac{v_m-w_m}{\rho}\right|\,dx\nonumber\\
&\qquad\qquad+c\int_{B\cap B_{2\rho}(y)}\eta^s
|\A_m(x,Dv_m)||Dw_m|\,dx=:I_1+I_2,
\end{align}
here $c\equiv c(n,\nu,L,s(G),s(K),s(H))$ is positive constant. We first estimate the term $I_1$. Proceeding exactly as in the proof of
Lemma~\ref{lem:higher_int_local}, for every $\varepsilon\in(0,1)$ we obtain
\begin{align}|I_1|&\le c\varepsilon\int_{B\cap B_{2\rho}(y)}\eta^s\bigl[\Psi(x,Dv_m)+\varepsilon_mK(|Dv_m|)\bigr]\,dx\nonumber\\
&\quad+c_\varepsilon\int_{B\cap B_{2\rho}(y)}\left[\Psi\!\left(x,\frac{v_m-w_m}{\rho}\right)+\varepsilon_m
K\!\left(\frac{|v_m-w_m|}{\rho}\right)\right]dx,
\end{align}
where
\[c\equiv c\bigl(n,s(G),s(H),s(K),\nu,L\bigr),\]
and
\[c_\varepsilon\equiv c_\varepsilon\bigl(
n,s(G),s(H),s(K),\nu,L,\varepsilon\bigr).\]

Similarly, applying Lemma~\ref{lem:Young epsilon} to the second integral gives
\begin{align}
|I_2|&\le c\varepsilon\int_{B\cap B_{2\rho}(y)}
\eta^s\bigl[\Psi(x,Dv_m)+\varepsilon_mK(|Dv_m|)\bigr]\,dx\nonumber\\
&\quad+c_\varepsilon\int_{B\cap B_{2\rho}(y)}
\bigl[\Psi(x,Dw_m)+\varepsilon_mK(|Dw_m|)\bigr]\,dx.\label{eq:17}
\end{align}

Choosing $\varepsilon>0$ sufficiently small and combining
\eqref{eq:18}--\eqref{eq:17}, we arrive at
\begin{align}
&\int_{B\cap B_{2\rho}(y)}\bigl[\Psi(x,Dv_m)+\varepsilon_mK(|Dv_m|)\bigr]\,dx\nonumber\\
&\qquad\le c\int_{B\cap B_{2\rho}(y)}
\left[\Psi\!\left(x,\frac{v_m-w_m}{\rho}
\right)+\varepsilon_m
K\!\left(\frac{|v_m-w_m|}{\rho}\right)\right]dx\nonumber\\
&\qquad\quad+c\int_{B\cap B_{2\rho}(y)}
\bigl[\Psi(x,Dw_m)+\varepsilon_mK(|Dw_m|)\bigr]\,dx.
\end{align}

Applying Theorem~\ref{Sobolev-Poincare}, we obtain
\begin{align*}
&\dashint_{B\cap B_{2\rho}(y)}\Psi\!\left(x,\frac{v_m-w_m}{\rho}\right)\,dx
\le c\left(\dashint_{B\cap B_{2\rho}(y)}[\Psi_1(x,Dv_m-Dw_m)]^{d_1}\,dx\right)^{\frac1{d_1}} \nonumber\\
&\quad\le c\left(\dashint_{B\cap B_{2\rho}(y)}[\Psi_1(x,Dv_m)]^{d_1}\,dx\right)^{\frac1{d_1}}
+c\dashint_{B\cap B_{2\rho}(y)}\Psi_1(x,Dw_m)\,dx,
\end{align*}
where $c$ and $d_1$ are the constants obtained in \eqref{eq:RH3}, with the additional dependence on $L_2$. Proceeding exactly as in the proof of Lemma~\ref{lem:higher_int_local}, we also obtain
\begin{align}
&\dashint_{B\cap B_{2\rho}(y)}[\Psi(x,Dv_m)+\varepsilon_mK(|Dv_m|)]\,dx\nonumber\\
&\quad\le c\left(\dashint_{B\cap B_{2\rho}(y)}[\Psi_1(x,Dv_m)+\varepsilon_mK(|Dv_m|)]^{\tilde d}\,dx\right)^{\frac1{\tilde d}} \nonumber\\
&\qquad+c\dashint_{B\cap B_{2\rho}(y)}[\Psi_1(x,Dw_m)+\varepsilon_mK(|Dw_m|)]\,dx,
\label{eq:15}
\end{align}
where
\[c\equiv c(L_1,L_2,n,s(G),s(H),s(K),\nu,L,\alpha,\beta,[a]_{0,\alpha},[b]_{0,\beta}).\]

If $B_{2\rho}(y)\Subset B$, then the argument is identical to the interior case treated in Lemma~\ref{lem:higher_int_local}. Otherwise, define
\begin{align*}
    V_m(x)&:=[\Psi(x,Dv_m)+\varepsilon_mK(|Dv_m|)]^{\tilde d}\chi_B(x),\\
U_m(x)&:=[\Psi_1(x,Dw_m)+\varepsilon_mK(|Dw_m|)]\chi_B(x).
\end{align*}

Then \eqref{eq:15} can be rewritten as
\begin{equation}
\dashint_{B_\rho(y)}V_m(x)^{1/\tilde d}\,dx
\le c\left(\dashint_{B_{2\rho}(y)}V_m(x)\,dx\right)^{1/\tilde d}
+c\dashint_{B_{2\rho}(y)}U_m(x)\,dx+c.
\label{eq:14}
\end{equation}

After a suitable rearrangement of \eqref{eq:14}, we apply a version of Gehring's lemma (see \cite[Theorem 6.6]{2003_Giusti_Direct_methods_in_the_calculus_of_variations}) and conclude that there exists $\delta_2\in(0,\delta_1]$ such that \eqref{eq:19} holds. This completes the proof.
\end{proof}
Now, we present the last theorem of this section, which eventually plays an important role in the proof of the main result of this paper.

\begin{theorem}\label{thm:fractional_diff}
Under the assumptions \eqref{coeff_1}--\eqref{condition_2}, \eqref{growth_conditions}, \eqref{eq:boundary1} and \eqref{eq:boundary2}, there exists a unique solution $w\in w_0+W^{1,\Psi}_0(B)$ to \eqref{eq:dirichlet} satisfying
\begin{equation}\label{eq:20}
G(|Dw|)\in L^{\frac{n}{n-2\varrho}}_{\mathrm{loc}}(B)\cap W^{\varrho,2}_{\mathrm{loc}}(B)
\end{equation}
for every $\varrho<\frac{\gamma}{2}$, together with the energy estimate
\begin{equation*}
\int_B\Psi(x,Dw)\,dx\le c\int_B\Psi(x,Dw_0)\,dx,
\end{equation*}
where $c\equiv c(n,s(G),s(H),\nu,L)$.

Moreover, the global higher integrability estimate
\begin{equation*}
\int_B[\Psi(x,Dw)]^{1+\sigma}\,dx\le c^*\int_B[\Psi(x,Dw_0)]^{1+\sigma}\,dx+c^*
\end{equation*}
holds for some constants $c^*,\,\sigma\equiv c^*,\,\sigma(n,s(G),s(H),\gamma,\nu,L_0,L_1),$ with $\sigma<\delta$.

Furthermore, the solution $w$ can be obtained as the limit of the sequence $\{v_m\}_{m=1}^\infty$ solving \eqref{eq:regularized-dirichlet}, in the sense that, up to a subsequence,
\begin{equation*}
v_m\rightharpoonup w \quad\text{in }W^{1,G^{1+\sigma}}(B),\qquad
v_m\to w \quad\text{in }W^{1,G^{\frac{n}{n-2\varrho}}}(B)
\end{equation*}
for every $\varrho<\frac{\gamma}{2}$. In particular, we can conclude that
\[
H(|Dw|)\in L^1_{\mathrm{loc}}(B),\qquad
G^{1+\frac{\gamma}{n}}(|Dw|)\in L^1_{\mathrm{loc}}(B).
\]
\end{theorem}

\begin{proof} The proof consists of five steps. We present the first three steps in detail, whereas the last two can be obtained by adapting the corresponding arguments from \cite[Theorem 5.1]{Baasandorj2020}.

\medskip
\noi\textbf{Step 1: Approximation.}
For each $m\in\mathbb{N}$ and $\sigma$ which we will define after a few lines, denote the quantity
\begin{equation}
\varepsilon_m:=
\left(
m+
\left(\int_B G^{1+\frac{\gamma}{n}}(|Dw_m|)\,dx\right)^3+
\left(\int_B G^{\left(1+\frac{\gamma}{n}\right)(1+\sigma)}(|Dw_m|)\,dx\right)^3
\right)^{-1}.
\label{eq:eps}
\end{equation}
Since $\{w_m\}_{m=1}^\infty\subset W^{1,\infty}(B)$, the sequence $\varepsilon_m$ is well defined and finite. Moreover, the following convergence holds
\[
\lim_{m\to\infty}\varepsilon_m
\int_B G^{1+\frac{\gamma}{n}}(|Dw_m|)\,dx=0.
\]

Throughout the proof, we set
\[
K(t):=G(t)^{1+\frac{\gamma}{n}},
\]
which is a Young function belonging to $C^1([0,\infty))\cap C^2((0,\infty))$. Also, we choose this $\varepsilon_m$ in \eqref{eq:regularized-vector-field}. Moreover, assumption \eqref{condition_2} yields
\[
\limsup_{t\to\infty}\frac{H(t)}{K(t)}\le \kappa<\infty.
\]

Let $v_m$ denote the solution of the regularized problem. Its weak formulation is given by
\begin{equation*}
    \int_B\langle \A_m(x,Dv_m),D\varphi\rangle\,dx=0
\qquad\text{for every }\varphi\in W^{1,K}_0(B).
\end{equation*}

By the corresponding energy estimate \eqref{seq_wm_vm}, we get
\begin{equation}
\int_B\bigl[\Psi(x,Dv_m)+\varepsilon_mK(|Dv_m|)\bigr]\,dx
\le
c\int_B\bigl[\Psi(x,Dw_m)+\varepsilon_mK(|Dw_m|)\bigr]\,dx,
\label{eq:energy_vm}
\end{equation}
where $c=c(n,s(G),s(H),\nu,L)$.

Since $w_m\to w_0$ in energy, for all sufficiently large $m$, we get
\[
\int_B\Psi(x,Dw_m)\,dx\le2\int_B\Psi(x,Dw_0)\,dx=:L_2.
\]
Combining this with \eqref{eq:eps} and \eqref{eq:energy_vm}, we obtain
\[
\int_B\Psi(x,Dv_m)\,dx\le c\int_B\Psi(x,Dw_0)\,dx+c=:L_1,
\]
where $c=c(n,s(G),s(H),\nu,L)$.

Therefore, Lemmas~\ref{lem:higher_int_local} and \ref{lem:higher_int_global} are applicable, yielding higher integrability exponents $\delta_1$ and $\delta_2$, with $\delta_2\le\delta_1$. We then choose
\[
0<\sigma<\min\left\{\frac{\delta_2}{2},\frac{\gamma}{n}\right\}.
\]

Applying Lemmas~\ref{lem:modular_equivalence} and \ref{lem:higher_int_global}, together with the definition of $\varepsilon_m$, gives
\begin{align}\label{eq:24}
    \int_BG^{1+\sigma}(|Dv_m|)\,dx&\le c\int_B\bigl[\Psi_1(x,Dv_m)+\varepsilon_mK(|Dv_m|)\bigr]^{1+\sigma}\,dx\nonum\\
&\le c\int_B\bigl[\Psi(x,Dw_m)+\varepsilon_mK(|Dw_m|)\bigr]^{1+\sigma}\,dx+c\\
&\le c\int_B[\Psi(x,Dw_0)]^{1+\sigma}\,dx+c,\nonum
\end{align}
where
\[c=c(n,s(G),s(H),\nu,L,\alpha,\beta,[a]_{0,\alpha},[b]_{0,\beta},L_0).\]

Hence, up to a subsequence, we conclude
\[v_m\rightharpoonup w\quad\text{in }W^{1,G^{1+\sigma}}(B)\]
for some
\[w\in w_0+W^{1,G^{1+\sigma}}_0(B).\]
Finally, by lower semicontinuity in \eqref{eq:24} and \eqref{eq:energy_vm}, yields
\[\int_B\Psi(x,Dw)\,dx\le c\int_B\Psi(x,Dw_0)\,dx\]
and
\[\int_B[\Psi(x,Dw)]^{1+\sigma}\,dx\le c^*\int_B[\Psi(x,Dw_0)]^{1+\sigma}\,dx+c^*\]
where
\[c=c(n,s(G),s(H),\nu,L),\qquad c^*=c^*(n,s(G),s(H),\nu,L,\alpha,\beta,[a]_{0,\alpha},[b]_{0,\beta},L_0).\]

\medskip
\noi\textbf{Step 2: Scaling and proof of \eqref{eq:20}.}
In this step we establish the fractional differentiability estimate. More precisely, let
\[\varrho\in\left(\frac{\gamma}{2(1+\sigma)},\frac{\gamma}{2}\right),\]
where $\sigma$ is chosen in Step~1. We claim that there exists a constant
\[c=c(n,s(G),s(H),\nu,L,\varrho)>0\]
such that
\begin{align}\label{eq:25}
&\left(\dashint_{B_{R/2}(x_0)}
G^{\frac{n}{n-2\varrho}}(|Dw|)\,dx
\right)^{\frac{n-2\varrho}{n}}\nonum\\
&\le
c\Bigl(\|a\|_{L^\infty(B)}+\|b\|_{L^\infty(B)}
+R^\alpha[a]_{0,\alpha;B}+R^\beta[b]_{0,\beta;B}\Bigr)
\dashint_{B_R(x_0)}G(|Dw|)\,dx  \nonum\\
&\quad+cR^{2\varrho+n(b_2-1)}
\Bigl(\|b\|_{L^\infty(B)}
+R^\beta[b]_{0,\beta;B}\Bigr)^{b_1}
R^{-2b_1\varrho}
\left(
\dashint_{B_R(x_0)}
G^{1+\sigma}(|Dw|)\,dx
\right)^{b_2}\nonum\\
&\quad+cR^{n-2\varrho},
\end{align}
where
\[
b_1=\frac{2\varrho(1+\sigma)-n\sigma}{2\varrho(1+\sigma)-\gamma}\ge1
\quad \text{ and }\quad
b_2=\frac{2\varrho(n+\gamma)-n\gamma}
{n\bigl(2\varrho(1+\sigma)-\gamma\bigr)}\ge1.
\]

To simplify the argument, we first reduce the proof to the unit ball by a scaling argument. Define
\[
\widetilde{w}(x):=\frac{w(x_0+Rx)}{R},\qquad \widetilde{w}_0(x):=\frac{w_0(x_0+Rx)}{R},
\]
together with
\[
\widetilde{w}_m(x):=\frac{w_m(x_0+Rx)}{R},\qquad\widetilde{v}_m(x):=\frac{v_m(x_0+Rx)}{R},
\]
and
\[
\widetilde{a}(x):=a(x_0+Rx) \text{ and } \widetilde{b}(x):=b(x_0+Rx)\qquad \text{ for } x\in B_1(0).
\]

The H\"older seminorm and the $L^\infty$ norm of the coefficient satisfy
\begin{align*}
[\widetilde{a}]_{0,\alpha;B_1}=R^\alpha[a]_{0,\alpha;B_R(x_0)},
\qquad
\|\widetilde{a}\|_{L^\infty(B_1)}=\|a\|_{L^\infty(B_R(x_0))},\\
[\widetilde{b}]_{0,\beta;B_1}=R^\beta[b]_{0,\beta;B_R(x_0)},
\qquad
\|\widetilde{b}\|_{L^\infty(B_1)}=\|b\|_{L^\infty(B_R(x_0))},
\end{align*}

while
\[\dashint_{B_1}G(|D\widetilde{w}|)\,dx=\dashint_{B_R(x_0)}G(|Dw|)\,dx.\]

Furthermore, setting
\[
\hat{\A}(x,z)=\A(x_0+Rx,z),\qquad\widetilde{\Psi}(x,z)=\widetilde{a}(x)G(|z|)+\widetilde{b}(x)H(|z|),
\]
we see that $\widetilde{w}$ solves the scaled Dirichlet problem
\[
-\operatorname{div}\hat{\A}(x,D\widetilde{w})=0
\quad\text{in }B_1(0),
\qquad\widetilde{w}\in\widetilde{w}_0+W^{1,\widetilde{\Psi}}_0(B_1(0)).
\]
Moreover, the structural assumptions remain unchanged under this scaling. Therefore, it is enough to establish \eqref{eq:25} in the case $R=1, x_0=0.$
The general estimate then follows by scaling back. In the next step, we derive uniform Nikolski estimates for the sequence $\{V_G(Dv_m)\}$.

\medskip
\noi\textbf{Step 3: Fractional Caccioppoli inequality and integrability}

We prove that, for every $\varrho\in(0,\gamma/2)$, there exists a constant $c>0$, depending only on $n$, $\kappa$, $s(G)$, $s(H)$, $\nu$, $L$ and $\varrho$, such that
\begin{align}\label{eq:26}
\|G(|Dv_m|)\|_{L^{\frac{n}{n-2\varrho}}(B_{1/2})}\le &c\bigl(\|a\|_{L^\infty(B_1)}+\|b\|_{L^\infty(B_1)}+[a]_{0,\alpha;B_1}+[b]_{0,\beta;B_1}+1\bigr)\nonum\\
&\qquad\times\|G(|Dv_m|)\|_{L^1(B_1)}\nonum\\
&+c\bigl(\|b\|_{L^\infty(B_1)}+[b]_{0,\beta;B_1}+\varepsilon_m\bigr)\|K(|Dv_m|)\|_{L^1(B_1)}+c.
\end{align}
To this end, we use the weak formulation of \eqref{eq:regularized-dirichlet} as follows
\begin{equation}\label{eq:weakvm}
\int_{B_1}\langle \A_m(x,Dv_m),D\varphi\rangle\,dx=0,\qquad \forall\,\varphi\in W^{1,K}_0(B_1).
\end{equation}
Choose the cut-off function $\eta\in C_0^\infty(B_{3/4})$ such that $\chi_{B_{2/3}}\le\eta\le\chi_{B_{3/4}}$ and $|D\eta|^2+|D^2\eta|\le10^4$. Fix $h\in\mathbb{R}^n$ with $|h|\le10^{-4}$ and set
\[
s(x,h):=|Dv_m(x+h)|+|Dv_m(x)|.
\]
Taking $\varphi=\tau_{-h}(\eta^2\tau_h(v_m))$ in \eqref{eq:weakvm} and using the integration-by-parts formula for finite differences, we obtain
\begin{align}\label{eq:fd}
\int_{B_1}\langle\tau_h\A_m(\cdot,Dv_m),D(\eta^2\tau_hv_m)\rangle\,dx=0.
\end{align}
Next, we write
\begin{align}\label{eq:decomp}
\tau_h\A_m(\cdot,Dv_m)=[\A_m(x+h,Dv_m(x+h))-\A_m(x+h,Dv_m(x))]\nonum\\+[\A_m(x+h,Dv_m(x))-\A_m(x,Dv_m(x))]=:J_1+J_2.
\end{align}
Substituting \eqref{eq:decomp} into \eqref{eq:fd}, we arrive at
\begin{align}
I_0&:=\int_{B_1}\eta^2\langle J_1,\tau_h(Dv_m)\rangle\,dx=-\int_{B_1}\eta^2\langle J_2,\tau_h(Dv_m)\rangle\,dx-2\int_{B_1}\eta\langle J_2,D\eta\rangle\tau_hv_m\,dx\nonumber\\
&\qquad\qquad\qquad\qquad\qquad\qquad-2\int_{B_1}\eta\langle J_1,D\eta\rangle\tau_hv_m\,dx=:I_1+I_2+I_3.\label{eq:27}
\end{align}
We estimate the terms $I_0$, $I_1$, $I_2$ and $I_3$ separately.

We begin with the left-hand side of \eqref{eq:27}. By \eqref{eq:Phi-mon} and \eqref{eq:A-monotonicity}, applied with $z_1=Dv_m(x+h)$ and $z_2=Dv_m(x)$ for $\Phi\equiv G,H,K$, we obtain
\begin{equation}\label{I0-lower}
\int_{B_1}\eta^2\Bigl[a(x+h)|\tau_h V_G(Dv_m)|^2+b(x+h)|\tau_h V_H(Dv_m)|^2+\varepsilon_m|\tau_h V_K(Dv_m)|^2\Bigr]dx\lesssim I_0.
\end{equation}
Next, using \eqref{growth_conditions}$_3$, $|a(x+h)-a(x)|\le[a]_{0,\alpha}|h|^\alpha$, $|b(x+h)-b(x)|\le[b]_{0,\beta}|h|^\beta$ and the monotonicity of $G',H'$, we have
\[
|J_2(x)|\;\lesssim\;[a]_{0,\alpha}|h|^\alpha\,G'(s(x,h))+[b]_{0,\beta}|h|^\beta\,H'(s(x,h)),
\]
so that $I_1=I_1^a+I_1^b$ with
\begin{align*}
|I_1^a|&\lesssim[a]_{0,\alpha}|h|^\alpha\int_{B_1}\eta^2G'(s(x,h))\,|\tau_h(Dv_m)|\,dx\quad \text{ and}\\
|I_1^b|&\lesssim[b]_{0,\beta}|h|^\beta\int_{B_1}\eta^2H'(s(x,h))\,|\tau_h(Dv_m)|\,dx.
\end{align*}

For $I_1^a$, Lemma~\ref{lem:Young epsilon} applied with $\Phi=G$ gives, for every $\varepsilon\in(0,1)$,
\[
G'(s(x,h))\,|\tau_h(Dv_m)|\;\le\;\varepsilon\,G(|\tau_h(Dv_m)|)+c_\varepsilon\,G(s(x,h)).
\]
Since, $G\in \N$, we can derive $$\int_{B_1}G(s(x,h))\,dx\le c\int_{B_1}G(|Dv_m|)\,dx.$$
Also, by Lemma~\ref{lem:basic-N} (with $\lambda=|h|\le1$) and Lemma~\ref{lem:finite-difference}, we have
\[
\int_{B_1}G(|\tau_h(Dv_m)|)\,dx\;\le\;|h|^{s(G)+1}\int_{B_1}G\!\left(\frac{|\tau_h(Dv_m)|}{|h|}\right)dx\;\le\;c\,|h|\int_{B_1}G(|Dv_m|)\,dx.
\]
Choosing $\varepsilon$ so that $\varepsilon|h|\le c_\varepsilon$, we conclude
\begin{equation}\label{I1a}
|I_1^a|\;\lesssim\;[a]_{0,\alpha}|h|^\alpha\int_{B_1}G(|Dv_m|)\,dx.
\end{equation}

For $I_1^b$, by the growth condition \eqref{condition_2}, Lemma~\ref{lem:Young epsilon} with $\Phi=K$ and proceeding as above, we get
\begin{equation}\label{I1b}
|I_1^b|\;\lesssim\;[b]_{0,\beta}|h|^\beta\int_{B_1}\bigl(G(|Dv_m|)+K(|Dv_m|)\bigr)\,dx.
\end{equation}
Combining \eqref{I1a} and \eqref{I1b}, we have
\begin{equation}\label{I1-final}
|I_1|\;\lesssim\;[a]_{0,\alpha}|h|^\alpha\int_{B_1}G(|Dv_m|)\,dx+[b]_{0,\beta}|h|^\beta\int_{B_1}\bigl(G(|Dv_m|)+K(|Dv_m|)\bigr)\,dx.
\end{equation}

\medskip
For $I_2.$
Using a splitting similar to the one above gives $I_2=I_2^a+I_2^b$. Using \eqref{growth_conditions}$_3$ and Lemma~\ref{lem:Young epsilon}, we get
\begin{align*}
|I_2^b|\lesssim[b]_{0,\beta}|h|^\beta\int_{B_1}\eta H'(|Dv_m|)\,|\tau_h(v_m)|\,dx\nonum\\
\lesssim[b]_{0,\beta}|h|^{\beta+1}\int_{B_1}\eta \left(H(|Dv_m|)
+H\!\left(\frac{|\tau_h(v_m)|}{|h|}\right)\right)\,dx.
\end{align*}
In the last display applying Lemma~\ref{lem:finite-difference} and then the growth condition \eqref{condition_2}, we have
\begin{align*}
|I_2^b|&\lesssim[b]_{0,\beta}|h|^{\beta+1}\int_{B_1}H(|Dv_m|)\,dx\nonum\\
&\lesssim[b]_{0,\beta}|h|^{\beta+1}\int_{B_1}\eta \left(H(|Dv_m|)
+K(|Dv_m|)\right)\,dx.
\end{align*}
Similarly, we get 
\begin{align*}
|I_2^a|\lesssim[a]_{0,\alpha}|h|^{\alpha+1}\int_{B_1}G(|Dv_m|)\,dx.
\end{align*}
Hence,
\begin{equation}\label{I2-final}
|I_2|\;\lesssim\;[a]_{0,\alpha}|h|^{\alpha+1}\int_{B_1}G(|Dv_m|)\,dx+[b]_{0,\beta}|h|^{\beta+1}\int_{B_1}\bigl(G(|Dv_m|)+K(|Dv_m|)\bigr)\,dx.
\end{equation}

\medskip
Finally, we estimate $I_3$, from \eqref{eq:A-difference}, we can get
\begin{align*}
|I_3|
&\le c\int_{B_1}\eta \Bigl(a(x+h)G'(s(x,h))+b(x+h)H'(s(x,h))\nonumber\\
&\qquad\qquad\qquad+\varepsilon_mK'(s(x,h))\Bigr)\frac{|\tau_hDv_m|}{s(x,h)}|\tau_hv_m|\,dx\nonumber\\
&\le c|h|\int_{B_1}\eta\frac{|\tau_hv_m|}{|h|}\Bigl(a(x+h)G'(s(x,h))+b(x+h)H'(s(x,h))\nonumber\\
&\qquad\qquad\qquad+\varepsilon_mK'(s(x,h))\Bigr)\,dx.
\end{align*}
Now, applying a similar approach as in \cite[page~32]{Baasandorj2020}, we obtain
\begin{equation}\label{I3-final}
|I_3|\;\lesssim\;|h|\bigl(\|a\|_{L^\infty}+\|b\|_{L^\infty}\bigr)\int_{B_1}G(|Dv_m|)\,dx+|h|\bigl(\|b\|_{L^\infty}+\varepsilon_m\bigr)\int_{B_1}K(|Dv_m|)\,dx.
\end{equation}

\medskip
Combining \eqref{I0-lower}, \eqref{I1-final}, \eqref{I2-final} and \eqref{I3-final} in \eqref{eq:27}, we are now in a position to derive the following desired fractional Caccioppoli type inequality
\begin{align}
&\int_{B_1}\eta^2\Bigl[a(x+h)|\tau_hV_G(Dv_m)|^2+b(x+h)|\tau_hV_H(Dv_m)|^2+\varepsilon_m|\tau_hV_K(Dv_m)|^2\Bigr]dx\nonum\\
&\,\,\lesssim\Bigl([a]_{0,\alpha}|h|^\alpha+[a]_{0,\alpha}|h|^{\alpha+1}+[b]_{0,\beta}|h|^\beta+[b]_{0,\beta}|h|^{\beta+1}\nonum\\
&\qquad \ +|h|\bigl(\|a\|_{L^\infty}+\|b\|_{L^\infty}\bigr)\Bigr)\int_{B_1}G(|Dv_m|)\,dx\nonum\\
&\qquad+\Bigl([b]_{0,\beta}|h|^\beta+[b]_{0,\beta}|h|^{\beta+1}+|h|\bigl(\|b\|_{L^\infty}+\varepsilon_m\bigr)\Bigr)\int_{B_1}K(|Dv_m|)\,dx.
\end{align}

Since $\gamma=\min\{\alpha,\beta\}$ gives $|h|^\alpha,|h|^\beta\le|h|^\gamma$ for $|h|\le1$, setting
\[
T:=\|a\|_{L^\infty(B_1)}+\|b\|_{L^\infty(B_1)}+[a]_{0,\alpha;B_1}+[b]_{0,\beta;B_1},
\qquad
T_m:=\|b\|_{L^\infty(B_1)}+[b]_{0,\beta;B_1}+\varepsilon_m,
\]
we obtain that
\begin{equation}
    \int_{B_1} \eta^2 \Big[ a(x+h)|\tau_h V_G(Dv_m)|^2 + b(x+h)|\tau_h V_H(Dv_m)|^2 \Big] dx \le \mathscr{C}_m |h|^\gamma,
\end{equation}
valid for every $|h|\le10^{-4}$, with $c\equiv c(n,\kappa,s(G),s(H),\nu,L)$. Here, for simplicity, we used the notation
\[
\mathscr{C}_m:= c \Big( T \int_{B_1} G(|Dv_m|) dx + T_m \int_{B_1} K(|Dv_m|) dx \Big).
\]
Dropping the non-negative terms yields
\begin{align}
\int_{B_1} \eta^2 a(x+h)|\tau_h V_G(Dv_m)|^2 dx &\le \mathscr{C}_m|h|^\gamma, \label{VG-est} \\
\int_{B_1} \eta^2 b(x+h)|\tau_h V_H(Dv_m)|^2 dx &\le \mathscr{C}_m|h|^\gamma. \label{VH-est}
\end{align}

Next, we aim to apply Lemma~\ref{lem:nikolski} to obtain \eqref{eq:26}. Using the H\"{o}lder continuity of $a(x)$ and $b(x)$, we define the radii
\[
r_a := \min\left\{\frac{1}{6},\left(\frac{\mu}{8[a]_{0,\alpha}}\right)^{1/\alpha}\right\}\quad \text{ and } \quad r_b := \min\left\{\frac{1}{6},\left(\frac{\mu}{8[b]_{0,\beta}}\right)^{1/\beta}\right\}.
\]
These radii depend only on the parameters $\mu, [a]_{0,\alpha}, [b]_{0,\beta}$ and are independent of the choice of the point $x_0$. 

Let $x_0 \in B_{1/2}$. If $a(x_0) \geq \mu/2$, then for any $x \in B_{r_a/2}(x_0)$ and $|h| < r_a/2$, then $x+h \in B_{r_a}(x_0)$. The H\"{o}lder continuity then yields
\begin{equation}\label{lower-bound-a}
a(x+h) \geq a(x_0)-[a]_{0,\alpha}|x+h-x_0|^\alpha \geq \frac{\mu}{2}-[a]_{0,\alpha}r_a^\alpha \geq \frac{\mu}{2}-\frac{\mu}{8} = \frac{3\mu}{8}.
\end{equation}
Analogously, if $b(x_0) \geq \mu/2$, it follows that $b(x+h) \geq 3\mu/8$ for all $x \in B_{r_b/2}(x_0)$ whenever $|h| < r_b/2$. Since $x_0 \in B_{1/2}$ and $r_a, r_b \leq 1/6$, the local balls $B_{r_a/2}(x_0)$ and $B_{r_b/2}(x_0)$ are strictly contained within $B_{2/3}$, ensuring the global cut-off $\eta \equiv 1$ throughout the domain of integration.

For any fixed $x_0 \in B_{1/2}$, the condition $a(x_0)+b(x_0) \geq \mu$ implies that at least one of the coefficients satisfies the lower bound $\mu/2$. Thus, we proceed by considering two sub-cases. 

\textit{Case} $(A)$: $b(x_0) \ge \mu/2$. Using the lower bound \eqref{lower-bound-a} for $b(x+h)$ and the fact that $\eta \equiv 1$ on $B_{r_b/2}(x_0)$. From \eqref{VH-est} for all $|h| < r_b/2$, we obtain 
\[
\int_{B_{r_b/2}(x_0)} |\tau_h V_H(Dv_m)|^2 dx \le \frac{8}{3\mu} \int_{B_1} \eta^2 b(x+h)|\tau_h V_H|^2 dx \le \frac{8\mathscr{C}_m}{3\mu}|h|^\gamma =: M_A^2 |h|^\gamma.
\]
To apply Lemma~\ref{lem:nikolski}, we set $\rho = r_b/4$, $R = r_b/2$ and $d = \gamma/2$. We choose a local cut-off $\eta_0 \in C_0^\infty(B_{3r_b/8}(x_0))$ such that $\chi_{B_{r_b/4}(x_0)} \le \eta_0 \le \chi_{B_{3r_b/8}(x_0)}$. Since $\eta_0 \le 1$, we have
\[
\int_{B_{r_b/2}(x_0)} \eta_0^2 |\tau_h V_H|^2 dx \le M_A^2 |h|^{2d} \qquad \text{for } |h| \le \frac{r_b}{16}.
\]
Applying Lemma~\ref{lem:nikolski} yields, for any $\varrho < \gamma/2$
\begin{equation}\label{eq:38}
 \|V_H(Dv_m)\|_{L^{2n/(n-2\varrho)}(B_{r_b/4}(x_0))} \le c\,r_b^{-(2\varrho+\gamma+2)} \Big( M_A + \|V_H(Dv_m)\|_{L^2(B_{r_b/2}(x_0))} \Big).   
\end{equation}
To bound the $L^2$ norm, we evaluate the bound $b(x) \ge 3\mu/8$ at $h=0$. This implies $$H(|Dv_m|) \le \frac{8}{3\mu}b(x)H(|Dv_m|) \le \frac{8}{3\mu}\Psi(x,Dv_m).$$ The uniform energy bound then gives
\[
\|V_H(Dv_m)\|^2_{L^2(B_{r_b/2}(x_0))} \approx \int_{B_{r_b/2}(x_0)} H(|Dv_m|) dx \le \frac{8}{3\mu} \int_{B_1} \Psi(x,Dv_m) dx \le \frac{8L_1}{3\mu}.
\]
Substituting this into the estimate \eqref{eq:38}, squaring both sides and using the equivalence $|V_H(z)|^2 \approx H(|z|)$, we find
\[
\int_{B_{r_b/4}(x_0)} H(|Dv_m|)^{n/(n-2\varrho)} dx \le c\,\mathcal{C}_A^{n/(n-2\varrho)},
\] where $\mathcal{C}_A := c\,r_b^{-(2\varrho+\gamma+2)}(\mathscr{C}_m+L_1+1).$
Finally, we use the inequality \eqref{eq:equivalence}, that is, $G \le C_0(H+1)$. Taking the power $p = n/(n-2\varrho) > 1$ gives
\begin{equation}\label{CaseA-final}
\int_{B_{r_b/4}(x_0)} G(|Dv_m|)^{n/(n-2\varrho)} dx \le c\,\mathcal{C}_A^{n/(n-2\varrho)} + c|B_{r_b/4}(x_0)| \le \hat{C}_A^{n/(n-2\varrho)},
\end{equation}
where $\hat{C}_A \le c(\mathscr{C}_m + L_1 + 1)$.

\textit{Case} $(B)$: $a(x_0) \ge \mu/2$. Following a parallel argument for \eqref{VG-est}, the bound $a(x+h) \ge 3\mu/8$ provides
\[
\int_{B_{r_a/2}(x_0)} |\tau_h V_G(Dv_m)|^2 dx \le \frac{8}{3\mu} \int_{B_1} \eta^2 a(x+h)|\tau_h V_G|^2 dx \le \frac{8\mathscr{C}_m}{3\mu}|h|^\gamma =: M_B^2 |h|^\gamma.
\]
Using Lemma~\ref{lem:nikolski} exactly as before with a cutoff supported in $B_{3r_a/8}(x_0)$, we obtain
\[
\|V_G(Dv_m)\|_{L^{2n/(n-2\varrho)}(B_{r_a/4}(x_0))} \le c\,r_a^{-(2\varrho+\gamma+2)} \Big( M_B + \|V_G(Dv_m)\|_{L^2(B_{r_a/2}(x_0))} \Big).
\]
By evaluating $a(x) \ge 3\mu/8$ at $h=0$, we bound $G(|Dv_m|) \le \frac{8}{3\mu}\Psi(x,Dv_m)$, bounding the local $L^2$ norm of $V_G$ uniformly by $c(L_1)$. As $|V_G(z)|^2 \approx G(|z|)$, yielding
\begin{equation}\label{CaseB-final}
\int_{B_{r_a/4}(x_0)} G(|Dv_m|)^{n/(n-2\varrho)} dx \le \hat{C}_B^{n/(n-2\varrho)},
\end{equation}
where $\hat{C}_B \le c(\mathscr{C}_m+L_1+1).$

Let $r_0 := \min\{r_a, r_b\}/4$. We can cover the compact set $\overline{B_{1/2}}$ with $N \le c(n)r_0^{-n}$ balls $\{B_{r_0/2}(x_j)\}_{j=1}^N$ centered at $x_j \in B_{1/2}$. Since $r_0/2 \le \min\{r_a, r_b\}/8$, each ball $B_{r_0/2}(x_j)$ is contained within both $B_{r_a/4}(x_j)$ and $B_{r_b/4}(x_j)$. Because either \textit{Case} $(A)$ or \textit{Case} $(B)$ must apply at each $x_j$, inequalities \eqref{CaseA-final} and \eqref{CaseB-final} guarantee
\[
\int_{B_{r_0/2}(x_j)} G(|Dv_m|)^{n/(n-2\varrho)} dx \le \tilde C^{n/(n-2\varrho)},
\] where $\tilde C := \max(\hat{C}_A, \hat{C}_B) \le c(\mathscr{C}_m+L_1+1).$ Summing over the finite subcover yields
\[
\int_{B_{1/2}} G(|Dv_m|)^{n/(n-2\varrho)} dx \le \sum_{j=1}^N \int_{B_{r_0/2}(x_j)} G(|Dv_m|)^{n/(n-2\varrho)} dx \le c\,r_0^{-n}\tilde C^{n/(n-2\varrho)}.
\]
Taking the $(n-2\varrho)/n$ power of both sides and noting that $r_0$ is independent of $m$, we obtain
\[
\|G(|Dv_m|)\|_{L^{n/(n-2\varrho)}(B_{1/2})} \le c(\mathscr{C}_m + L_1 + 1).
\]
Substituting the definition of $\mathscr{C}_m$ gives
\[
\|G(|Dv_m|)\|_{L^{n/(n-2\varrho)}(B_{1/2})} \le c T\|G(|Dv_m|)\|_{L^1(B_1)} + c T_m\|K(|Dv_m|)\|_{L^1(B_1)} + c.
\]
This proves \eqref{eq:26}.

The remainder of the proof follows exactly as in Step-$4$ and Step-$5$ of the proof of \cite[Theorem 5.1]{Baasandorj2020}. We therefore omit the details.
\end{proof}

\section{Higher integrability estimates}\label{sec_6}
\noi The next result provides a local higher integrability estimate for the energy density. This lemma will be used to obtain a key estimate in Step~$6$ of the proof of Theorem~\ref{main_thm}.
\begin{theorem}\label{thm:6.1}
Let $u\in W^{1,\Psi}(\Omega)$ be a distributional solution of \eqref{model_problem_2} satisfying assumptions \eqref{growth_conditions}, \eqref{growth_condition_2}, \eqref{condition_1} and \eqref{condition_2}. Suppose, in addition, that $\Psi(\cdot,F)\in L^\Theta_{\mathrm{loc}}(\Omega)$ for some $\Theta\in\N$ with $s(\Theta)\ge1$. Then there exists a constant $\delta=\delta(\textnormal{\texttt{data}})>0$, satisfying $\delta<1/s(\Theta)$, such that
\[
\Psi_1(\cdot,Du)\in L^{1+\delta}_{\mathrm{loc}}(\Omega).
\]
Furthermore, there exists a constant $c=c(\textnormal{\texttt{data}})>0$ for which
\begin{equation}\label{eq:36}
\left(\dashint_{B_\rho}[\Psi_1(x,Du)]^{1+\delta}\,dx\right)^{\frac{1}{1+\delta}}
\le c\dashint_{B_{2\rho}}\Psi_1(x,Du)\,dx
+c\left(\dashint_{B_{2\rho}}[\Psi(x,F)]^{1+\delta}\,dx\right)^{\frac{1}{1+\delta}}
\end{equation}
holds for every ball $B_{2\rho}\subset\Omega$ with $2\rho\le1$. In the homogeneous case $F\equiv0$, for every open set $\Omega_0\Subset\Omega$ there exists a constant $c=c(\textnormal{\texttt{data}},\operatorname{dist}(\Omega_0,\partial\Omega))>0$ such that
\begin{equation}
\|\Psi_1(\cdot,Du)\|_{L^{1+\delta}(\Omega_0)}\le c.
\end{equation}
\end{theorem}

\begin{proof}
Let $B_{2\rho}\Subset\Omega$ and choose a cut-off function $\eta\in C_0^\infty(B_{2\rho})$ satisfying $\chi_{B_\rho}\le\eta\le\chi_{B_{2\rho}}$ and $|D\eta|\le\frac{4}{\rho}.$
For $\varphi=\eta^s\bigl(u-(u)_{B_{2\rho}}\bigr), s\ge \max\{s(G),s(H)\}+1,$
Lemma~\ref{lem:existence} ensures that $\varphi$ is an admissible test function for the weak formulation of \eqref{model_problem_2}. Proceeding as in the proof of Lemma~\ref{lem:higher_int_local} and using \eqref{growth_conditions}, \eqref{growth_condition_2}, together with Young's inequality, we obtain
\begin{equation}
\int_{B_{2\rho}}\Psi(x,Du)\eta^s\,dx
\le
c\int_{B_{2\rho}}\Psi\left(x,\frac{u-(u)_{B_{2\rho}}}{\rho}\right)\,dx
+c\int_{B_{2\rho}}\Psi(x,F)\,dx,
\end{equation}
where $c=c(n,s(G),s(H),\nu,L)$. Adding $1$ both sides, we get
\begin{equation}
\int_{B_{2\rho}}\Psi_1(x,Du)\eta^s\,dx
\le
c\int_{B_{2\rho}}\Psi_1\left(x,\frac{u-(u)_{B_{2\rho}}}{\rho}\right)\,dx
+c\int_{B_{2\rho}}\Psi(x,F)\,dx,
\end{equation}

Applying Theorem~\ref{Sobolev-Poincare} to the first term on the right-hand side yields
\begin{equation}
\dashint_{B_\rho}\Psi_1(x,Du)\,dx
\le
c\left(\dashint_{B_{2\rho}}[\Psi_1(x,Du)]^d\,dx\right)^{\frac1d}
+c\dashint_{B_{2\rho}}\Psi(x,F)\,dx,
\label{eq:37}
\end{equation}
where $c=c(\texttt{data})$ and $d=d(n,s(G),s(H),\gamma)\in(0,1)$.

In order to apply a variant of Gehring's lemma, we start by proving that $\Psi(x,F)\in L^{1+\frac1{s(\Theta)}}(\Omega_0)$. Thus, for every $\Omega_0\Subset\Omega$, we estimate
\begin{align*}
\int_{\Omega_0}[\Psi(x,F)]^{1+\frac1{s(\Theta)}}\,dx
&=
\int_{\Omega_0\cap\{\Psi(x,F)\le1\}}[\Psi(x,F)]^{1+\frac1{s(\Theta)}}\,dx\\
&\qquad\qquad+\int_{\Omega_0\cap\{\Psi(x,F)>1\}}[\Psi(x,F)]^{1+\frac1{s(\Theta)}}\,dx\\
&\le
|\Omega_0|
+c\int_{\Omega_0}\Theta(\Psi(x,F))\,dx
<\infty,
\end{align*}
where the last inequality follows from Lemma~\ref{lem:basic-N} applied to the Young function $\Theta$.

Finally, a suitable version of Gehring's lemma applied to \eqref{eq:37} yields the existence of $\delta=\delta(\texttt{data})>0$, with $\delta<\frac1{s(\Theta)},$
such that \eqref{eq:36} holds.
\end{proof}
\begin{remark}
    Let $u$ be the original solution to \eqref{model_problem_2}. By Theorem \ref{thm:6.1} and a standard covering argument, for every choice of open subset $\Omega_0\Subset\Omega_1\Subset\Omega$ there exists a constant $c\equiv c(\textnormal{\texttt{data}}_0)$ and $0<\delta<\frac{1}{s(\Theta)}$ which depends on $\textnormal{\texttt{data}}$ such that
$$\|\Psi_1(\cdot,Du)\|_{L^{1+\delta}(\Omega_0)}\le c(\textnormal{\texttt{data}}_0).$$
In particular, since $\Psi_1:=\Psi+1$, we get
\begin{equation}\label{eq:energy_bound}
    \|\Psi(\cdot,Du)\|_{L^{1+\delta}(\Omega_0)}\le c(\textnormal{\texttt{data}}_0),
\end{equation}
where $\textnormal{\texttt{data}}_0\equiv(\textnormal{\texttt{data}}, \dist(\Omega_1,\partial\Omega),\|\Psi_1(\cdot,F)\|_{L^\Theta(\Omega_1)}).$
\end{remark}
\section{Proof of Theorem \ref{main_thm}: exit time arguments and conclusion}\label{sec_7}

\noi The proof of Theorem~\ref{main_thm} proceeds in ten steps. For the sake of consistency, we adopt the same structural outline presented in the proof of \cite[Theorem~1.1]{Colombo2016} and \cite[Theorem~2.1]{Baasandorj2020}.

\medskip
\noi\textbf{Step 1: Exit time and covering of the level set.} Fix a ball $B_R\equiv B_R(x_0)$ as in the statement of Theorem \ref{main_thm},
with $R\le r$ and $B_R\Subset\Omega_0$, where \(r\) is the radius appearing in the statement of the Theorem~\ref{main_thm}. The choice of \(r\) will be specified at the end of the proof and depends only on the constants introduced during the proof. We choose two auxiliary radii $r_1$ and $r_2$ such that
$\tfrac{R}{2}\le r_1<r_2\le R$. For $\lambda>1$ and $\tfrac{R}{2}\le s\le R,$
we introduce the level sets
\begin{equation*}
    E^s_\lambda:=\bigl\{x\in B_s(x_0):\Psi_1(x,Du)>\lambda\bigr\}.
\end{equation*}
Let $M\ge1$ be a constant depending only on \texttt{data}, to be fixed at the very end
of the proof. For every ball $B_\rho(y)\subset B_R$ define the average
\begin{equation*}
    \mathcal{F}\bigl(B_\rho(y)\bigr)
    :=\dashint_{B_\rho(y)}\bigl[\Psi_1(x,Du)+M\Psi_1(x,F)\bigr]\,dx.
\end{equation*}
By the Lebesgue differentiation theorem, for a.e.\ $y_0\in E^{s}_\lambda$
one has
\begin{equation}\label{step1-lebesgue}
    \lim_{\rho\to0^+}\mathcal{F}\bigl(B_\rho(y_0)\bigr)>\lambda,
    \qquad \tfrac{R}{2}\le s\le R.
\end{equation}
On the other hand, if $y_0\in B_{r_1}(x_0)$ and
$\rho\in\left[\tfrac{r_2-r_1}{80},\,r_2-r_1\right]$, then
$B_\rho(y_0)\subset B_{r_2}(x_0)$ and we get
\begin{equation}\label{step1-lambda0}
    \mathcal{F}\bigl(B_\rho(y_0)\bigr)
    \le \left(\frac{80\,r_2}{r_2-r_1}\right)^{\!n}
    \dashint_{B_{r_2}(x_0)}\bigl[\Psi_1(x,Du)+M\Psi_1(x,F)\bigr]\,dx
    =:\lambda_0.
\end{equation}
We can observe that $\lambda_0>1,$ since $\frac{r_2}{r_2-r_1}>1.$ From now on we only consider values $\lambda>\lambda_0$.

Combining \eqref{step1-lebesgue} and \eqref{step1-lambda0} the intermediate value
theorem yields an \emph{exit time} radius
$\rho_{y_0}\in\bigl(0,\tfrac{r_2-r_1}{80}\bigr)$ such that
\begin{equation}\label{step1-stopping}
    \mathcal{F}\bigl(B_{\rho_{y_0}}(y_0)\bigr)=\lambda,
    \qquad
    \mathcal{F}\bigl(B_\rho(y_0)\bigr)<\lambda
    \quad\text{for every }\rho\in(\rho_{y_0},\,r_2-r_1].
\end{equation}

Since \eqref{step1-stopping} holds for a.e.\ $y_0\in E^{r_1}_\lambda$,
the family of balls $\{B_{\rho_{y_0}}(y_0)\}$ covers $E^{r_1}_\lambda$ up to a null set. Vitali's covering theorem then provides a countable,
pairwise disjoint subfamily
$\{\widetilde{B}_i\}_i:=\{B_{\rho_{y_i}}(y_i)\}_i$ such that
\begin{equation*}
    E^{r_1}_\lambda\subset\bigcup_i 5\widetilde{B}_i,
    \qquad \{\widetilde{B}_i\}_i \text{ is mutually disjoint},
\end{equation*}
and 
\begin{equation}
    \mathcal{F}\bigl(B_{\rho_{y_i}}(y_i)\bigr)=\lambda,
    \qquad
    \mathcal{F}\bigl(B_\rho(y_i)\bigr)<\lambda
    \quad\text{for every }\rho\in(\rho_{y_i},\,r_2-r_1].
\end{equation}
Moving forward by setting 
\begin{equation}\label{eq:32}
    B_i\equiv 5B_{\rho_{y_i}}(y_i) \text{ and } \rho_i:=5\rho_{y_i},
\end{equation} 
one checks
directly from the construction that
\begin{equation}\label{step1-geometry}
    80\widetilde{B}_i=16B_i\subset B_{r_2}(x_0),
    \qquad
    \rho_i\le\frac{r_2-r_1}{16}\le r\le1,
\end{equation}
and recalling \eqref{step1-stopping}, we write
\begin{equation}\label{step1-final}
    \begin{cases}
        \displaystyle \mathcal{F}(\widetilde{B}_i) = \dashint_{\widetilde{B}_i} [\Psi_1(x, Du) + M\Psi_1(x, F)] \, dx = \lambda, \\[12pt]
        \displaystyle \mathcal{F}(16B_i) = \dashint_{16B_i} [\Psi_1(x, Du) + M\Psi_1(x, F)] \, dx \le \lambda.
    \end{cases}
\end{equation}
The family $\{B_i\}$, together with the identities
\eqref{step1-geometry}--\eqref{step1-final}, is what we use in the
subsequent comparison steps.

\medskip
\noi\textbf{Step 2: A first comparison estimate.} For each ball $B_i$ constructed in Step~1, we have $16B_i\Subset\Omega_0$. We consider the Dirichlet problem
\begin{equation}\label{step1_prob}
\begin{cases}
-\operatorname{div}\A(x,Du_i)=0 & \text{in }16B_i,\\
u_i\in u+W^{1,\Psi}_0(16B_i).
\end{cases}
\end{equation}
By Theorem~\ref{thm:fractional_diff}, there exists a unique solution $u_i$ satisfying
\begin{equation*}
u_i\in W^{1,G^{1+\frac{\gamma}{n}}}_{\mathrm{loc}}(16B_i),
\qquad
G(|Du_i|)\in L^{\frac{n}{n-2\varrho}}_{\mathrm{loc}}(16B_i)
\quad\text{for every }\varrho<\frac{\gamma}{2},
\end{equation*}
together with the estimates
\begin{equation}\label{eq:fist_comp}
\begin{aligned}
\int_{16B_i}\Psi(x,Du_i)\,dx
&\le
c_1\int_{16B_i}\Psi(x,Du)\,dx,\\
\int_{16B_i}[\Psi(x,Du_i)]^{1+\sigma}\,dx
&\le
c_2\int_{16B_i}[\Psi(x,Du)]^{1+\sigma}\,dx+c_2,
\end{aligned}
\end{equation}
where the constants $c_1\equiv  c_1(n,\nu,L,s(G),s(H))
\text{ and } c_2\equiv c_2(\texttt{data})$ and $\sigma\equiv\sigma(\texttt{data})$. We claim that for every $\varepsilon\in(0,1)$ there exists a constant $c_\varepsilon\equiv c_\ve(n,s(G),s(H),\nu,L,\varepsilon)$
such that
\begin{align}\label{eq:7.30}
\dashint_{16B_i}\Bigl(a(x)|V_G(Du_i)-V_G(Du)|^2+b(x)|V_H(Du_i)-V_H(Du)|^2\Bigr)\,dx\nonum\\
\qquad\le
\varepsilon\dashint_{16B_i}\Psi_1(x,Du)\,dx
+c_\varepsilon\dashint_{16B_i}\Psi_1(x,F)\,dx.
\end{align}
To verify this estimate, we choose $\varphi=u-u_i$
as a test function in the weak formulation of the equation satisfied by $u$, namely
\begin{equation*}
\int_{16B_i}\langle \A(x,Du)-\A(x,Du_i),D\varphi\rangle\,dx=\int_{16B_i}\langle \B(x,F),D\varphi\rangle\,dx.
\end{equation*}
The admissibility of this choice follows from \eqref{eq:fist_comp}, since $u-u_i\in W^{1,\Psi}_0(16B_i)$. Using the monotonicity and growth assumptions on $\A$, together with \eqref{eq:fist_comp}, we estimate
\begin{align}
&\dashint_{16B_i}\Bigl(a(x)|V_G(Du_i)-V_G(Du)|^2+b(x)|V_H(Du_i)-V_H(Du)|^2\Bigr)\,dx \nonumber\\
&\le c\dashint_{16B_i}\Bigl(a(x)G'(|F|)+b(x)H'(|F|)\Bigr)(|Du_i|+|Du|)\,dx \nonumber\\
&\le \varepsilon\dashint_{16B_i}\bigl[\Psi(x,Du_i)+\Psi(x,Du)\bigr]\,dx
+c_\varepsilon\dashint_{16B_i}\Psi(x,F)\,dx \nonumber\\
&\le 2\varepsilon\dashint_{16B_i}\Psi_1(x,Du)\,dx
+c_\varepsilon\dashint_{16B_i}\Psi_1(x,F)\,dx,
\label{eq:first_est}
\end{align}
where $c=c(n,s(G),s(H),\nu,L)$ and
$c_\varepsilon=c_\varepsilon(n,s(G),s(H),\nu,L,\varepsilon)$. In the second inequality, we have used Young's inequality, while the last one follows from \eqref{eq:fist_comp}. After replacing $\varepsilon$ by $\varepsilon/2$, estimate \eqref{eq:7.30} follows.

\medskip
\noi\textbf{Step 3: A second comparison estimate.} Let $x_{i,a}^m\in \overline{8B_i}$ be such that
\[a(x_{i,a}^m)=\inf_{x\in \overline{8B_i}}a(x)\text{ and define }
\Psi_a(x,z):=a(x_{i,a}^m)G(|z|)+b(x)H(|z|).
\]
Again applying Theorem~\ref{thm:fractional_diff}, there exists a unique function $w_i\in u_i+W^{1,\Psi_a}_0(8B_i)$
satisfying
\begin{equation}\label{eq:step3-problem}
\begin{cases}
-\operatorname{div}\A_a(x,Dw_i)=0 & \text{in }8B_i,\\
w_i\in u_i+W^{1,\Psi_a}_0(8B_i),
\end{cases}
\end{equation}
where $\A_a(x,z)=\tilde\A(a(x_{i,a}^m),b(x),z)$. Moreover,
\[
w_i\in W^{1,G^{1+\frac{\gamma}{n}}}_{\mathrm{loc}}(8B_i),\qquad
G(|Dw_i|)\in L^{\frac{n}{n-2\varrho}}_{\mathrm{loc}}(8B_i)
\quad\text{for every }\varrho<\frac{\gamma}{2},
\]
and
\begin{equation}\label{eq:second_comp}
   \dashint_{8B_i}\Psi_a(x,Dw_i)\,dx\le\dashint_{8B_i}\Psi_a(x,Du_i)\,dx. 
\end{equation}
We next verify that $\Psi(\cdot,Dw_i)\in L^1(8B_i),$
which guarantees that the forthcoming test function is admissible. Notice that using \eqref{step1_prob} the right hand side of \eqref{eq:step3-problem} is finite. Since $\Psi(x,z)=\Psi_a(x,z)
+\bigl(a(x)-a(x_{i,a}^m)\bigr)G(|z|),$
it is enough to prove that $G(|Dw_i|)\in L^1(8B_i).$ We distinguish two cases.

\medskip
\noi\textit{Case} $(a).$ Suppose that $a(x_{i,a}^m)>0.$
Then $\Psi_a(x,z)\ge a(x_{i,a}^m)G(|z|),$
and therefore
\[
\int_{8B_i}G(|Dw_i|)\,dx\le\frac{1}{a(x_{i,a}^m)}\int_{8B_i}\Psi_a(x,Dw_i)\,dx
<\infty.\]
\medskip
\noi \textit{Case} $(b).$ Assume that $a(x_{i,a}^m)=0.$ Since $a(x)+b(x)\ge\mu,$ we have $b(x_{i,a}^m)\ge\mu.$ Using the H\"older continuity of $b$, together with the choice of the radius in Step~1, we obtain $b(x)\ge\frac{\mu}{2}\text{ for every }x\in8B_i.$
Hence $\Psi_a(x,z)\ge\frac{\mu}{2}H(|z|).$
By the assumption $G(t)\le c_0(H(t)+1),$
we infer
\[
G(|Dw_i|)
\le
\frac{2c_0}{\mu}\Psi_a(x,Dw_i)+c_0,
\]
and consequently
\[
\int_{8B_i}G(|Dw_i|)\,dx
\le
\frac{2c_0}{\mu}
\int_{8B_i}\Psi_a(x,Dw_i)\,dx
+c|8B_i|
<\infty.
\]
Combining the two cases yields $G(|Dw_i|)\in L^1(8B_i).$
Moreover, H\"older continuity of $a(\cdot)$ implies $|a(x)-a(x_{i,a}^m)|
\le
[a]_{0,\alpha}(8\rho_i)^\alpha$ and therefore
\[\int_{8B_i}\Psi(x,Dw_i)\,dx
\le
\int_{8B_i}\Psi_a(x,Dw_i)\,dx
+[a]_{0,\alpha}(8\rho_i)^\alpha
\int_{8B_i}G(|Dw_i|)\,dx
<\infty.\]
Thus, $\Psi(\cdot,Dw_i)\in L^1(8B_i)$
and hence $w_i-u_i\in W^{1,\Psi}_0(8B_i).$

Since $w_i-u_i\in W^{1,\Psi}_0(8B_i),$ it is admissible as a test function. Subtracting the weak formulations satisfied by $u_i$ and $w_i$, we obtain
\begin{equation*}
\dashint_{8B_i}\langle \A_a(x,Dw_i)-\A_a(x,Du_i),D\varphi\rangle\,dx=
\dashint_{8B_i}\langle \A(x,Du_i)-\A_a(x,Du_i),D\varphi\rangle\,dx,
\end{equation*}
where $\varphi=w_i-u_i$.

Using the monotonicity of $\A_a$, we infer
\begin{align*}
&\dashint_{8B_i}\Bigl(a(x_{i,a}^m)|V_G(Du_i)-V_G(Dw_i)|^2
+b(x)|V_H(Du_i)-V_H(Dw_i)|^2\Bigr)\,dx \nonumber\\
&\qquad\le
c\dashint_{8B_i}
\langle \A(x,Du_i)-\A_a(x,Du_i),D(w_i-u_i)\rangle\,dx.
\end{align*}
From \eqref{growth_conditions}, we have
\[
\A(x,z)-\A_a(x,z)
\le
\bigl(a(x)-a(x_{i,a}^m)\bigr)
G'(|z|)\frac{z}{|z|}.
\]
Thus, we obtain
\begin{align*}
&\dashint_{8B_i}\Bigl(a(x_{i,a}^m)|V_G(Du_i)-V_G(Dw_i)|^2
+b(x)|V_H(Du_i)-V_H(Dw_i)|^2\Bigr)\,dx \nonumber\\
&\qquad\le
c\,\omega_a(8\rho_i)
\dashint_{8B_i}
G'(|Du_i|)
|D(w_i-u_i)|\,dx,
\end{align*}
where $\omega_a(8\rho_i):=\sup_{\substack{x,y\in8B_i}}|a(x)-a(y)|.$ Applying Young's inequality gives
\[
G'(|Du_i|)|D(w_i-u_i)|\le\varepsilon G(|D(w_i-u_i)|)+c_\varepsilon G(|Du_i|).
\]
Using the $\Delta_2$-condition for $G$, we have
\[
G(|D(u_i-w_i)|)
\le
c\bigl(G(|Du_i|)+G(|Dw_i|)\bigr),
\]
and the estimate $G(t)\le c\bigl(\Psi(x,t)+1\bigr),$
we deduce
\begin{align*}
&\dashint_{8B_i}\Bigl(a(x_{i,a}^m)|V_G(Du_i)-V_G(Dw_i)|^2
+b(x)|V_H(Du_i)-V_H(Dw_i)|^2\Bigr)\,dx \nonumber\\
&\qquad\le
c\,\omega_a(8\rho_i)
\dashint_{8B_i}
\bigl(\Psi_{a,1}(x,Dw_i)+\Psi_1(x,Du_i)\bigr)\,dx,
\end{align*}
where $c\equiv c(n,s(G),s(H),\mu).$ Finally, using \eqref{eq:fist_comp}, \eqref{eq:second_comp}
and the H\"older continuity of $a$,
we conclude that
\begin{align}\label{eq:step3-final}
\dashint_{8B_i}\Bigl(a(x_{i,a}^m)|V_G(Du_i)-V_G(Dw_i)|^2
+b(x)&|V_H(Du_i)-V_H(Dw_i)|^2\Bigr)\,dx\nonum\\
&\qquad\le
\hat c\,\rho_i^\alpha
\dashint_{16B_i}\Psi_1(x,Du)\,dx,
\end{align}
here $\hat c\equiv \hat c(n,s(G),s(H),\mu).$

\medskip
\noi\textbf{Step 4: A third comparison estimate.} Let $x_{i,b}^M\in \overline{2B_i}$ be such that
\[b(x_{i,b}^M)=\sup_{x\in 2B_i}b(x)
\text{ and set }
\Psi_{a,b}(z):=a(x_{i,a}^m)G(|z|)+b(x_{i,b}^M)H(|z|).
\]
By Theorem~\ref{thm:fractional_diff} with $a(x)\equiv a(x_{i,a}^m)$ and $ b(x)\equiv b(x_{i,b}^M)$, there exists a unique solution $v_i\in w_i+W_0^{1,\Psi_{a,b}}(2B_i)$
such that
\begin{equation}\label{eq:step4-problem}
\begin{cases}
-\operatorname{div}\A_{a,b}(x,Dv_i)=0 & \text{in }2B_i,\\
v_i\in w_i+W^{1,\Psi_{a,b}}_0(2B_i).
\end{cases}
\end{equation}
where $\A_{a,b}(x,z)=\tilde\A(a(x_{i,a}^m),b(x_{i,b}^M),z)$. Moreover,
\begin{equation}\label{eq:third_comp}
    \dashint_{2B_i}\Psi_{a,b}(Dv_i)\,dx\le\dashint_{2B_i}\Psi_{a,b}(Dw_i)\,dx.
\end{equation}
We claim that the right-hand side is finite. Suppose that $a(x_{i,a}^m)=0$. Then $b(x_{i,b}^M)\ge \mu,$
and therefore
\[
\Psi_{a,b}(z)=b(x_{i,b}^M)H(|z|)\ge \mu H(|z|).
\]
Since $\ds\int_{2B_i}\Psi_a(x,Dw_i)\,dx<\infty,$
we conclude that
\[\int_{2B_i}H(|Dw_i|)\,dx\le
\frac{1}{\mu}\int_{2B_i}\Psi_a(x,Dw_i)\,dx<\infty.
\]
 Suppose that $a(x_{i,a}^m)>0$. Since $w_i$ is a unique solution of \eqref{eq:step3-problem}, we may apply \cite[Theorem~5.1]{Baasandorj2020} to the rescaled integrand
\[\frac1{a(x_{i,a}^m)}\Psi_a(x,z)=G(|z|)+\frac{b(x)}{a(x_{i,a}^m)}H(|z|),\]
which yields $H(|Dw_i|)\in L^1(2B_i).$ Hence, in either case, $\ds \int_{2B_i}H(|Dw_i|)\,dx<\infty.$
Consequently,
\[
\int_{2B_i}\Psi_{a,b}(Dw_i)\,dx=a(x_{i,a}^m)\int_{2B_i}G(|Dw_i|)\,dx+
b(x_{i,b}^M)\int_{2B_i}H(|Dw_i|)\,dx<\infty.
\]
Now, we prove that $v_i\in W^{1,H}(2B_i)$, i.e., $H(|Dv_i|)\in L^1(2B_i).$ If $b(x_{i,b}^M)>0$, then
\[
b(x_{i,b}^M)\int_{2B_i}H(|Dv_i|)\,dx
\le
\int_{2B_i}\Psi_{a,b}(Dv_i)\,dx\le \int_{2B_i}\Psi_{a,b}(Dw_i)\,dx<\infty,
\]
whereas if $b(x_{i,b}^M)=0$, then $b\equiv0$ in $2B_i$, so that $\Psi_{a,b}\equiv\Psi_a\text{ in }2B_i,$
and therefore $v_i=w_i\text{ in }2B_i.$
Since $H(|Dw_i|)\in L^1(2B_i)$, we again conclude that $H(|Dv_i|)\in L^1(2B_i).$

Hence $v_i-w_i\in W_0^{1,\Psi_{a,b}}(2B_i)$, so it can be used as a test function. Subtracting the weak formulations for $v_i$ and $w_i$, we obtain
\[
\dashint_{2B_i}\bigl\langle \A_{a,b}(Dv_i)-\A_{a,b}(Dw_i),D\varphi\bigr\rangle
=
\dashint_{2B_i}\bigl\langle \A_a(x,Dw_i)-\A_{a,b}(Dw_i),D\varphi\bigr\rangle
\]
for $\varphi=v_i-w_i$. Now, proceeding in a similar way to Step 3, we conclude
\begin{align}
&\dashint_{2B_i}\Bigl(a(x_{i,a}^m)|V_G(Dv_i)-V_G(Dw_i)|^2
+b(x_{i,b}^M)|V_H(Dv_i)-V_H(Dw_i)|^2\Bigr)\,dx
\nonumber\\
&\qquad\qquad\le
c\,\omega_b(2\rho_i)
\dashint_{2B_i}
H'(|Dw_i|)|D(v_i-w_i)|\,dx:=I.
\label{eq:step4-final}
\end{align}

\medskip
\noi\textbf{Step 5: Estimate in the $(G,H)$-phase.} First, we estimate $I$ in the $(G,H)$-phase, that is, for $K_0\ge 4,$ we have
$$\inf_{x\in2B_i}b(x)>K_0[b]_{0,\beta}\rho_i^\beta.$$ Then, we observe that
\begin{equation*}
\operatorname*{osc}_{2B_i}b\le4[b]_{0,\beta}\rho_i^\beta\le\frac{4b(x)}{K_0}
\qquad\text{for every }x\in2B_i.
\end{equation*}
Consequently,
\begin{equation}\label{eq:7.44}
b(x_{i,b}^M)\le b(x)+\operatorname*{osc}_{2B_i}b\le b(x)+4[b]_{0,\beta}\rho_i^\beta\le b(x)+\frac{4b(x)}{K_0}\le2b(x).
\end{equation}

Using the convexity of $H$, we estimate
\begin{align}
I&\le\frac{c}{K_0}\dashint_{2B_i}b(x)H'(|Dw_i|)|Dv_i-Dw_i|\,dx\nonumber\\
&\le\frac{c}{K_0}\dashint_{2B_i}b(x)H(|Dw_i|+|Dv_i|)\,dx\nonumber\\
&\le\frac{c}{K_0}\dashint_{2B_i}\bigl[\Psi_{a,b}(Dv_i)+b(x)H(|Dw_i|)\bigr]\,dx\nonumber\\
&\overset{\eqref{eq:third_comp}}{\le}\frac{c}{K_0}\dashint_{2B_i}\bigl[\Psi_{a,b}(Dw_i)+b(x_{i,b}^M)H(|Dw_i|)\bigr]\,dx\nonumber\\
&\overset{\eqref{eq:7.44}}{\le}\frac{c}{K_0}\dashint_{2B_i}\Psi_a(x,Dw_i)\,dx\nonumber\\
&\overset{\eqref{eq:fist_comp},\eqref{eq:second_comp}}{\le}\frac{c}{K_0}\dashint_{16B_i}\Psi_1(x,Du)\,dx,
\label{eq:7.45}
\end{align}
Here $c\equiv c(n,s(G),s(H),\nu,L).$ Substituting \eqref{eq:7.45} into \eqref{eq:step4-final}, we obtain
\begin{align}\label{eq:G_H_phase}
\dashint_{2B_i}\Bigl(a(x_{i,a}^m)|V_G(Dw_i)-V_G(Dv_i)|^2
+b(x_{i,b}^M)|V_H(Dw_i)-V_H(Dv_i)|^2\Bigr)\,dx\nonum\\
\le
\frac{c_1}{K_0}\dashint_{16B_i}\Psi_1(x,Du)\,dx,
\end{align}
where $c_1\equiv c_1(n,s(G),s(H),\nu,L).$ Furthermore, estimate \eqref{eq:7.45} also yields
\begin{equation}\label{eq:G_H_final}
\int_{2B_i}\Psi_{a,b}(Dw_i)\,dx
\le c\int_{16B_i}\Psi_1(x,Du)\,dx,
\end{equation}
where $c\equiv c(n,s(G),s(H),\nu,L).$

\medskip
\noi\textbf{Step 6: Estimate in the $G$-phase.} Now, we estimate $I$ in the $G$-phase, that is, for $K_0\ge 4$, we consider
$$\inf_{x\in2B_i}b(x)\le K_0[b]_{0,\beta}\rho_i^\beta.$$
 By the H\"older continuity of \(b\), we can write
\begin{equation}\label{7.48}
\begin{aligned}
b(x_{i,b}^M)=\sup_{2B_i}b(x)\le 4[b]_{0,\beta}\rho_i^\beta+\inf_{2B_i}b(x)\le (4+K_0)[b]_{0,\beta}\rho_i^\beta.
\end{aligned}
\end{equation}
Since \(a(x)+b(x)\ge\mu\) in \(8B_i\) and, by the H\"older continuity of $b$,
\[
\sup_{8B_i}b(x)\le\inf_{2B_i}b(x)+\operatorname*{osc}_{8B_i}b\le K_0[b]_{0,\beta}\rho_i^\beta+[b]_{0,\beta}(16\rho_i)^\beta\le(16+K_0)[b]_{0,\beta}\rho_i^\beta,
\]
recalling the definition of $x_{i,a}^m$ in Step~3, we get
\[
a(x_{i,a}^m)=\inf_{\overline{8B_i}}a(x)\ge\mu-\sup_{8B_i}b(x)\ge\mu-(16+K_0)[b]_{0,\beta}\rho_i^\beta.
\]
Choosing \(R_*>0\) such that \((16+K_0)[b]_{0,\beta}R_*^\beta\le\mu/2\), we obtain
\(a(x_{i,a}^m)\ge\mu/2\) whenever \(\rho_i\le R_*\). Therefore \(w_i\) solves the Euler-Lagrange equation corresponding to the normalized functional
\[
\frac{1}{a(x_{i,a}^m)}\Psi_a(x,z)
=G(|z|)+\widetilde b(x)H(|z|),\qquad
\widetilde b(x):=\frac{b(x)}{a(x_{i,a}^m)},
\]
where \([\widetilde b]_{0,\beta}\le2[b]_{0,\beta}/\mu\). Hence, we can apply the conditional reverse H\"older inequality \cite[Theorem~7.1]{Baasandorj2020}. Taking
\[
s=1+\frac{\gamma}{n}\le1+\frac{\beta}{n}<\frac{n}{n-\beta},
\]
we obtain
\begin{equation}\label{7.49}
\dashint_{2B_i}G^{1+\frac{\gamma}{n}}(|Dw_i|)\,dx
\le
c\left(\dashint_{8B_i}\Psi_a(x,Dw_i)\,dx\right)^{1+\frac{\gamma}{n}},
\end{equation}
where $c\equiv c\!\left(\texttt{data},K_0\right).$
Moreover, \(c\) is nondecreasing with respect to the \(L^1(8B_i)\)-norm of \(\Psi_a(\cdot,Dw_i)\). By \eqref{eq:fist_comp} and \eqref{eq:second_comp}, we get
\[
\|\Psi_a(\cdot,Dw_i)\|_{L^1(8B_i)}
\le
c\|\Psi_a(\cdot,Du)\|_{L^1(8B_i)}\le
c\|\Psi_1(\cdot,Du)\|_{L^1(16B_i)},
\]
therefore \(c\) depends only on the data. Since \(\operatorname{osc}_{2B_i}b\le b(x_{i,b}^M)\), Young's inequality gives, for every \(\tau\in(0,1)\),
\begin{align}
I
&\le
cb(x_{i,b}^M)\dashint_{2B_i}H'(|Dw_i|)|Dw_i|\,dx
+c\dashint_{2B_i}H'(|Dw_i|)|Dv_i|\,dx \nonumber\\
&\le
c\dashint_{2B_i}b(x_{i,b}^M)H(|Dw_i|)\,dx
+\tau\dashint_{2B_i}b(x_{i,b}^M)H(|Dv_i|)\,dx\nonumber\\
&\qquad
+\frac{c}{\tau^{s(H)}}\dashint_{2B_i}b(x_{i,b}^M)H(|Dw_i|)\,dx \nonumber\\
&\le
c\left(1+\frac{1}{\tau^{s(H)}}\right)\dashint_{2B_i}b(x_{i,b}^M)H(|Dw_i|)\,dx
+\tau\dashint_{2B_i}b(x_{i,b}^M)H(|Dv_i|)\,dx,
\label{7.50}
\end{align}
where \(c\equiv c(n,s(G),s(H),\nu,L)\).

We next estimate the first term on the right-hand side of \eqref{7.50}. By \eqref{7.48}, we have
\begin{align}
&\dashint_{2B_i}b(x_{i,b}^M)H(|Dw_i|)\,dx
\le
c\rho_i^\beta\dashint_{2B_i}\left(G(|Dw_i|)+G^{1+\frac{\gamma}{n}}(|Dw_i|)\right)\,dx \nonumber\\
&\quad\overset{\eqref{7.49}}{\le}
c\rho_i^\beta\dashint_{2B_i}\Psi_{a,1}(x,Dw_i)\,dx
+c\rho_i^\beta\left(\dashint_{8B_i}\Psi_{a,1}(x,Dw_i)\,dx\right)^{1+\frac{\gamma}{n}} \nonumber\\
&\quad\le
c\left(\rho_i^\beta+\rho_i^\beta\left(\dashint_{8B_i}\Psi_{a,1}(x,Dw_i)\,dx\right)^{\frac{\gamma}{n}}\right)
\dashint_{8B_i}\Psi_{a,1}(x,Dw_i)\,dx .
\label{7.51a}
\end{align}
Applying H\"older's inequality, we obtain
\begin{align*}
    \dashint_{8B_i}\Psi_{a,1}(x,Dw_i)\,dx&=\dashint_{8B_i}\Psi_{a}(x,Dw_i)\,dx+c\\
   & \overset{\eqref{eq:second_comp}}{\le}\dashint_{8B_i}\Psi_{a}(x,Du_i)\,dx+c\\
   &\overset{\text{H\"older}}{\le}\left(\dashint_{8B_i}(\Psi_{a}(x,Du_i))^{1+\delta}\,dx\right)^{\frac1{1+\delta}}+c.
\end{align*}
Using the comparison estimate for \(u_i\), it follows that
\[
\dashint_{8B_i}\Psi_{a,1}(x,Dw_i)\,dx
\overset{\eqref{eq:fist_comp}}{\le}
\left(c_2\dashint_{8B_i}\Psi_a(x,Du)^{1+\delta}\,dx+c\right)^{\frac1{1+\delta}}+c.
\]
Since \(1/(1+\delta)<1\), an elementary inequality, we get
\[
\dashint_{8B_i}\Psi_{a,1}(x,Dw_i)\,dx
\le
c\left(\dashint_{8B_i}\Psi_a(x,Du)^{1+\delta}\,dx\right)^{\frac1{1+\delta}}+c.
\]
Moreover, by Theorem~\ref{thm:6.1}, we obtain
\[
\int_{8B_i}\Psi_a(x,Du)^{1+\delta}\,dx\le\int_{\Omega_1}\Psi_a(x,Du)^{1+\delta}\,dx \overset{\eqref{eq:energy_bound}}{\le}c(\mathrm{\texttt{data}}_0),
\]
and therefore
\[
\left(\dashint_{8B_i}\Psi_a(x,Du)^{1+\delta}\,dx\right)^{\frac1{1+\delta}}
\le
c\rho_i^{-\frac{n}{1+\delta}}.
\]
Consequently,
\[
\dashint_{8B_i}\Psi_{a,1}(x,Dw_i)\,dx
\le
c\rho_i^{-\frac{n}{1+\delta}}+c.
\]
Substituting this estimate into \eqref{7.51a}, we obtain
\[
\dashint_{2B_i}b(x_{i,b}^M)H(|Dw_i|)\,dx
\le
c\left(\rho_i^\beta+\rho_i^{\beta-\frac{\gamma}{1+\delta}}\right)
\dashint_{8B_i}\Psi_{a,1}(x,Dw_i)\,dx.
\]
Since \(\rho_i\le1\), setting
\[
s_1:=\min\left\{\beta,\frac{(\beta-\gamma)+\beta\delta}{1+\delta}\right\}=\frac{(\beta-\gamma)+\beta\delta}{1+\delta}>0,
\]
we conclude that
\begin{equation}\label{7.51}
\dashint_{2B_i}b(x_{i,b}^M)H(|Dw_i|)\,dx
\le
c\rho_i^{s_1}
\dashint_{8B_i}\Psi_{a,1}(x,Dw_i)\,dx .
\end{equation}

We now estimate the second term on the right-hand side of \eqref{7.50}. We argue as follows
\begin{align}
\dashint_{2B_i}b(x_{i,b}^M)H(|Dv_i|)\,dx
&\le
c\dashint_{2B_i}\Big(a(x_{i,a}^m)G(|Dv_i|)+b(x_{i,b}^M)H(|Dv_i|)\Big)\,dx \nonumber\\
&\overset{\eqref{eq:third_comp}}{\le}
c\dashint_{2B_i}\Big(a(x_{i,a}^m)G(|Dw_i|)+b(x_{i,b}^M)H(|Dw_i|)\Big)\,dx \nonumber\\
&\overset{\eqref{eq:equivalence}, \eqref{7.51}}{\le}
c\dashint_{8B_i}\Psi_{a,1}(x,Dw_i)\,dx
+c\rho_i^{s_1}\dashint_{8B_i}\Psi_{a,1}(x,Dw_i)\,dx \nonumber\\
&\le
c(1+\rho_i^{s_1})\dashint_{8B_i}\Psi_{a,1}(x,Dw_i)\,dx.
\label{7.52}
\end{align}
Substituting \eqref{7.51} and \eqref{7.52} into \eqref{7.50}, we find
\begin{equation*}
I\le
c\left(\left(1+\frac{1}{\tau^{s(H)}}\right)\rho_i^{s_1}
+\tau(1+\rho_i^{s_1})\right)
\dashint_{8B_i}\Psi_{a,1}(x,Dw_i)\,dx.
\end{equation*}
Choosing \(\tau=\rho_i^{\frac{s_1}{2s(H)}}\), we obtain
\begin{equation}\label{eq:28}
I\le
c\rho_i^{s_0}
\dashint_{16B_i}\Psi_{1}(x,Du)\,dx,
\end{equation}
where \(s_0:=\frac{s_1}{2s(H)}\). Here we have used that \(s(H)\ge1\) and \(\rho_i\le1\).

Combining \eqref{eq:step4-final} and \eqref{eq:28}, we conclude that
\begin{align}\label{eq:G_phase}
\dashint_{2B_i}\!\left(a(x_{i,a}^m)|V_G(Dv_i)-V_G(Dw_i)|^2
+b(x_{i,b}^M)|V_H(Dv_i)-V_H(Dw_i)|^2\right)\,dx\nonum\\
\le
c_0\rho_i^{s_0}
\dashint_{16B_i}\Psi_1(x,Du)\,dx,
\end{align}
where \(c_0\equiv c_0(\texttt{data}_0,K_0)\).

Finally, using \eqref{7.52},\eqref{eq:fist_comp} together with \eqref{eq:second_comp}, we obtain
\begin{equation}\label{eq:G_final}
\dashint_{2B_i}\Psi_{a,b}(Dw_i)\,dx
\le
c\dashint_{16B_i}\Psi_1(x,Du)\,dx,
\end{equation}
where \(c\equiv c(\mathrm{\texttt{data}}_0,K_0)\).

\medskip
\noi\textbf{Step 7: Matching the two phases and comparison estimates.}
Combining \eqref{eq:G_H_phase} and \eqref{eq:G_phase}, we obtain
\begin{align}
&\dashint_{2B_i}\Bigl(a(x_{i,a}^m)|V_G(Dw_i)-V_G(Dv_i)|^2+b(x)|V_H(Dw_i)-V_H(Dv_i)|^2\Bigr)\,dx \nonumber\\
&\qquad\le
\dashint_{2B_i}\Bigl(a(x_{i,a}^m)|V_G(Dw_i)-V_G(Dv_i)|^2+b(x_{i,b}^M)|V_H(Dw_i)-V_H(Dv_i)|^2\Bigr)\,dx \nonumber\\
&\qquad\le
\left(\frac{c_1}{K_0}+c_0\rho_i^{s_0}\right)
\dashint_{16B_i}\Psi_1(x,Du)\,dx.
\label{eq:both_phases}
\end{align}
Here $c_1\equiv c_1(n,s(G),s(H),\nu,L)$ and $c_0\equiv c_0(\texttt{data}_0,K_0)$, while $K_0\ge4$ will be chosen later and $s_0$ is introduced in \eqref{eq:28}. Now, using Lemma~\ref{lem:modular_equivalence}, \eqref{eq:G_H_final} and \eqref{eq:G_final}, we also have
\begin{align}
&\dashint_{2B_i}|V_G(Du_i)-V_G(Dv_i)|^2\,dx\le2\dashint_{2B_i}\Bigl(|V_G(Du_i)|^2+|V_G(Dv_i)|^2\Bigr)\,dx \nonumber\\
&\qquad\le c^*(K_0)\dashint_{16B_i}\Psi_1(x,Du)\,dx,
\label{eq:29}
\end{align}
where $c^*(K_0)\equiv c^*(\texttt{data}_0,K_0)$. Combining \eqref{eq:both_phases}, \eqref{eq:29} and \eqref{eq:step3-final} with \eqref{eq:first_est}, for every $\varepsilon\in(0,1)$, we obtain
\begin{align}\label{eq:combined_est}
\dashint_{2B_i}\Bigl(a(x)|V_G(Dv_i)-V_G(Du)|^2+b(x)|V_H(Dv_i)-V_H(Du)|^2\Bigr)\,dx\nonum\\
\le
\left(2\varepsilon+2c_0r^{s_0}+2(c^*(K_0)+\hat c)r^\alpha+\frac{2c_1}{K_0}\right)
\dashint_{16B_i}\Psi_1(x,Du)\,dx
+2c_\varepsilon\dashint_{16B_i}\Psi_1(x,F)\,dx.
\end{align}
Again, $c_1\equiv c_1(n,s(G),s(H),\nu,L)$, $c_0\equiv c_0(\texttt{data}_0,K_0)$, $c_\varepsilon\equiv c_\varepsilon(n,s(G),s(H),\nu,L,\varepsilon)$ and $s_0$ is given by \eqref{eq:28}. Now, we introduce the notation
\begin{equation}\label{eq:defn_S}
\mathcal S(\varepsilon,r,K_0,M):=
2\varepsilon+2c_0r^{s_0}+2(c^*(K_0)+\hat c)r^\alpha+\frac{2c_1}{K_0}+\frac{2c_\varepsilon}{M}.
\end{equation}
Using \eqref{step1-final}$_2$ in \eqref{eq:combined_est}, we deduce that, for every $K_0\ge4$, the following estimate follows
\begin{equation}\label{eq:31}
\dashint_{2B_i}\Bigl(a(x)|V_G(Du)-V_G(Dv_i)|^2+b(x)|V_H(Du)-V_H(Dv_i)|^2\Bigr)\,dx
\le
\mathcal S(\varepsilon,r,K_0,M)\lambda.
\end{equation}
This estimate holds for every ball $B_i$ in the covering constructed in Step~1 and is independent of whether the condition of $G$-phase or $(G,H)$-phase is satisfied.

Finally, we claim that
\begin{equation}
\dashint_{2B_i}\Bigl(a(x_{i,a}^m)G(|Dv_i|)+b(x_{i,b}^M)H(|Dv_i|)+1\Bigr)\,dx
\le
c\lambda,
\end{equation}
where $c\equiv c(\texttt{data}_0)$. Indeed,
\begin{align}\label{eq:step_7_final}
\dashint_{2B_i}\Bigl(\Psi_{a,b}(Dv_i)+1\Bigr)\,dx&\le
c\dashint_{2B_i}\bigl(\Psi_{a,b}(Dw_i)+1\bigr)\,dx\nonum\\
&\le
c\dashint_{16B_i}\Psi_1(x,Du)\,dx
\overset{\eqref{step1-final}_2}{\le}
c\lambda,
\end{align}
where the first inequality follows from \eqref{eq:third_comp}, while the second one follows from \eqref{eq:G_H_final} if \(\inf_{x\in2B_i}b(x)>K_0[b]_{0,\beta}\rho_i^\beta\) and from \eqref{eq:G_final} if \(\inf_{x\in2B_i}b(x)\le K_0[b]_{0,\beta}\rho_i^\beta\).

\medskip
\noi\textbf{Step 8: A priori estimate for $Dv_i$.} First, we apply \cite[Theorem~1.2]{Lieberman1991} to the problem \eqref{eq:step4-problem} satisfied by $v_i$ and obtain the Lipschitz estimate
\begin{equation}\label{eq:30}
   \sup_{x\in B_i}\Psi_{a,b}(Dv_i)+1\le
c\dashint_{2B_i}\bigl(\Psi_{a,b}(Dv_i)+1\bigr)\,dx
\overset{\eqref{eq:step_7_final}}{\le} c\lambda.
\end{equation}
Now, recall that $\Psi(x,Dv_i)=a(x)G(|Dv_i|)+b(x)H(|Dv_i|).$
Then using Lemma~\ref{lem:modular_equivalence}, we estimate as follows
\begin{align*}
a(x)G(|Dv_i|)
&\le \Bigl(\operatorname*{osc}_{2B_i}a+a(x_{i,a}^m)\Bigr)G(|Dv_i|)\\
&\le \bigl(\operatorname*{osc}_{2B_i}a\bigr)G(|Dv_i|)
+a(x_{i,a}^m)G(|Dv_i|)\\
&\le c\bigl(\Psi_{a,b}(Dv_i)+1\bigr)
+\Psi_{a,b}(Dv_i)\\
&\le c\bigl(\Psi_{a,b}(Dv_i)+1\bigr).
\end{align*}
Moreover, by the definition of $x_{i,b}^M$, we have
\[
b(x)H(|Dv_i|)
\le b(x_{i,b}^M)H(|Dv_i|)
\le \Psi_{a,b}(Dv_i).
\]
Hence, $\Psi(x,Dv_i)
\le c\bigl(\Psi_{a,b}(Dv_i)+1\bigr),$ which implies
\begin{equation}\label{estimate_v_i}
   \sup_{x\in B_i}\Psi(x,Dv_i)\le
c\sup_{x\in B_i}\bigl(\Psi_{a,b}(Dv_i)+1\bigr)
\overset{\eqref{eq:30}}{\le} c_\ell\lambda, 
\end{equation}
where $c_\ell\equiv c_\ell(\texttt{data}_0).$

\medskip
\noi\textbf{Step 9: Estimates involving level sets.} Using \eqref{eq:Phi-growth} together with an elementary inequality, we obtain
\begin{align}
&2c_\ell\lambda (s(G)+s(H))\left|B_i\cap\left\{\Psi_1(x,Du)>4(s(G)+s(H))c_\ell\lambda\right\}\right|\nonum\\
&\qquad+\frac12\int_{B_i\cap\{\Psi_1(x,Du)>4(s(G)+s(H))c_\ell\lambda\}}\Psi_1(x,Du)\,dx\nonumber\\
&\le\int_{B_i\cap\{\Psi_1(x,Du)>4(s(G)+s(H))c_\ell\lambda\}}\Psi_1(x,Du)\,dx\nonumber\\
&\le\frac{s(G)+s(H)}{1+s(G)+s(H)}
\int_{B_i\cap\{\Psi_1(x,Du)>4(s(G)+s(H))c_\ell\lambda\}}
\hspace{-2em}\bigl(a(x)|V_G(Du)|^2+b(x)|V_H(Du)|^2\bigr)\,dx\nonumber\\
&\le\frac{2(s(G)+s(H))}{1+s(G)+s(H)}
\int_{B_i}\bigl(a(x)|V_G(Du)-V_G(Dv_i)|^2+b(x)|V_H(Du)-V_H(Dv_i)|^2\bigr)\,dx\nonumber\\
&\qquad+2(s(G)+s(H))
\int_{B_i\cap\{\Psi_1(x,Du)>4(s(G)+s(H))c_\ell\lambda\}}
\Psi_1(x,Dv_i)\,dx\nonumber\\
&\le2(s(G)+s(H))
\int_{B_i}\bigl(a(x)|V_G(Du)-V_G(Dv_i)|^2+b(x)|V_H(Du)-V_H(Dv_i)|^2\bigr)\,dx\nonumber\\
&\qquad+2c_\ell\lambda(s(G)+s(H))
\left|B_i\cap\left\{\Psi_1(x,Du)>4(s(G)+s(H))c_\ell\lambda\right\}\right|,
\end{align}
where the last inequality follows from \eqref{estimate_v_i}. Hence,
\begin{align}
&\int_{B_i\cap\{\Psi_1(x,Du)>4(s(G)+s(H))c_\ell\lambda\}}\Psi_1(x,Du)\,dx\nonum\\
&\qquad\le4|2B_i|\dashint_{2B_i}\bigl(a(x)|V_G(Du)-V_G(Dv_i)|^2+b(x)|V_H(Du)-V_H(Dv_i)|^2\bigr)\,dx.
\end{align}

Using \eqref{eq:31}, \eqref{eq:32} and the identity \(|2B_i|=10^n|\widetilde B_i|\), we obtain
\begin{equation}\label{eq:33}
\int_{B_i\cap\{\Psi_1(x,Du)>4(s(G)+s(H))c_\ell\lambda\}}\Psi_1(x,Du)\,dx
\le40^n\mathcal S(\varepsilon,r,K_0,M)\lambda|\widetilde B_i|.
\end{equation}
Next, we want to estimate \(|\widetilde B_i|\). From \eqref{step1-final}, we have
\begin{equation}
|\widetilde B_i|=\frac1\lambda\int_{\widetilde B_i}\bigl(\Psi_1(x,Du)+M\Psi_1(x,F)\bigr)\,dx.
\end{equation}
We split the integral according to the corresponding level sets to obtain
\begin{align}
|\widetilde B_i|&\le\frac1\lambda
\int_{\widetilde B_i\cap\{\Psi_1(x,Du)>\lambda/4\}}\Psi_1(x,Du)\,dx\nonum\\
&\quad\qquad+\frac1\lambda\int_{\widetilde B_i\cap\{\Psi_1(x,F)>\lambda/(4M)\}}M\Psi_1(x,F)\,dx+\frac12|\widetilde B_i|,
\end{align}
which gives
\begin{equation}\label{eq:34}
|\widetilde B_i|\le\frac2\lambda\int_{\widetilde B_i\cap\{\Psi_1(x,Du)>\lambda/4\}}\Psi_1(x,Du)\,dx+\frac2\lambda
\int_{\widetilde B_i\cap\{\Psi_1(x,F)>\lambda/(4M)\}}M\Psi_1(x,F)\,dx.
\end{equation}
Combining \eqref{eq:33} and \eqref{eq:34}, we arrive at
\begin{align}\label{eq35}
&\int_{B_i\cap\{\Psi_1(x,Du)>4(s(G)+s(H))c_\ell\lambda\}}\Psi_1(x,Du)\,dx\nonumber\\
&\qquad\le80^n\mathcal S(\varepsilon,r,K_0,M)
\int_{\widetilde B_i\cap\{\Psi_1(x,Du)>\lambda/4\}}\Psi_1(x,Du)\,dx\nonumber\\
&\qquad\quad+80^n\mathcal S(\varepsilon,r,K_0,M)
\int_{\widetilde B_i\cap\{\Psi_1(x,F)>\lambda/(4M)\}}M\Psi_1(x,F)\,dx.
\end{align}

Since \(\{B_i\}\) covers \(E_\lambda^{r_1}\) and \(E_{4(s(G)+s(H))c_\ell\lambda}^{r_1}\subset E_\lambda^{r_1}\), summing over \(i\) yields
\begin{equation*}
\int_{E^{r_1}_{4(s(G)+s(H))c_\ell\lambda}}\Psi_1(x,Du)\,dx\le\sum_i\int_{B_i\cap\{\Psi_1(x,Du)>4(s(G)+s(H))c_\ell\lambda\}}\Psi_1(x,Du)\,dx.
\end{equation*}

For convenience, we introduce the notation
\begin{equation*}
\mathcal D_\lambda^s:=\{x\in B_s(x_0):\Psi_1(x,F(x))>\lambda\},\qquad \frac{R}{2}\le s\le R,\ \lambda>0.
\end{equation*}

Finally, using the fact that the balls \(\{\widetilde B_i\}\) are pairwise disjoint together with \eqref{step1-geometry}, summing \eqref{eq35} over \(i\) gives
\begin{align*}
\int_{E^{r_1}_\lambda}\Psi_1(x,Du)\,dx
&\le80^n\mathcal S(\varepsilon,r,K_0,M)
\int_{E^{r_2}_{\lambda/(16(s(G)+s(H))c_\ell)}}\Psi_1(x,Du)\,dx\nonumber\\
&\quad+80^n\mathcal S(\varepsilon,r,K_0,M)
\int_{\mathcal D^{r_2}_{\lambda/(16(s(G)+s(H))c_\ell M)}}M\Psi_1(x,F)\,dx,
\end{align*}
for every
\begin{align*}
\lambda\ge\lambda_1&:=4(s(G)+s(H))c_\ell\lambda_0\\
&=\frac{4(s(G)+s(H))c_\ell80^n\,r_2^n}{(r_2-r_1)^n}
\dashint_{B_{r_2}}\bigl[\Psi_1(x,Du)+M\Psi_1(x,F)\bigr]\,dx.
\end{align*}

\medskip
\noi\textbf{Step 10: Conclusion.} To conclude the proof, we employ a standard truncation argument. For $t\ge1,$ we define the truncated function by
\begin{equation*}
[\Psi_1(x,Du)]_t:=\min\{\Psi_1(x,Du),t\}.
\end{equation*}
Proceeding as in Step 10 of \cite[Theorem~1.1]{Baasandorj2020}, we obtain
\begin{align}\label{final_est}
&\dashint_{B_{r_1}}\Theta'([\Psi_1(x,Du)]_t)\Psi_1(x,Du)\,dx\nonum\\
&\le
c_f^{\,s(\Theta)+1}c_\ell^{\,s(\Theta)}
\mathcal S(\varepsilon,r,K_0,M)
\dashint_{B_{r_2}}
\Theta'([\Psi_1(x,Du)]_t)\Psi_1(x,Du)\,dx
\nonumber\\
&\qquad+c_f^{\,s(\Theta)+1}c_\ell^{\,s(\Theta)}M^{s(\Theta)}\mathcal S(\varepsilon,r,K_0,M)\dashint_{B_{r_2}}\Theta(\Psi_1(x,F))\,dx\nonumber\\
&\qquad+c_m\,c_f^{\,s(\Theta)+1}c_\ell^{s(\Theta)}\Theta(\lambda_0),
\end{align}
where \(c_\ell\equiv c_\ell(\texttt{data}_0)\) is defined in \eqref{estimate_v_i}, while
\(c_f\equiv c_f(n,s(G),s(H))\). This estimate is valid for every
\(M\ge1\), \(K_0\ge4\), \(r\le1\) and \(\varepsilon\in(0,1)\). We choose
\(\varepsilon\), \(r\), \(K_0\) and \(M\) so that
\begin{equation}\label{small_est}
c_f^{\,s(\Theta)+1}c_\ell^{\,s(\Theta)}\mathcal S(\varepsilon,r,K_0,M)\le\frac12.
\end{equation}
First, we set
\begin{equation}\label{const_est_1}
K_0:=16c_f^{\,s(\Theta)+1}c_\ell^{\,s(\Theta)+1}c_1,\qquad\varepsilon:=\frac{1}
{16c_f^{\,s(\Theta)+1}c_\ell^{\,s(\Theta)+1}},
\end{equation}
where \(c_1\) is the constant appearing in \eqref{eq:defn_S}. Consequently,
\(c_0\) and \(c_\varepsilon\) in \eqref{eq:defn_S} depend only on
\(\texttt{data}_0\) and \(s(\Theta)\). Next, we choose
\begin{equation}\label{const_est_2}
M:=16c_f^{\,s(\Theta)+1}c_\ell^{\,s(\Theta)+1}c_\varepsilon,
\end{equation}
which again depends only on \(\texttt{data}_0\) and \(s(\Theta)\). Finally, we choose \(r=r(\texttt{data}_0,s(\Theta))\) sufficiently small so that
\begin{equation}\label{const_est_3}
r\le \min \left\{ \left(\frac{1}{16c_f^{\,s(\Theta)+1}c_\ell^{\,s(\Theta)+1}c_0}\right)^{1/s_0}, \left(\frac{1}{16(\hat c+c^*(K_0))}\right)^{1/\alpha}, R_* \right\},
\end{equation}
where \(s_0\) is defined in \eqref{eq:28} and \(R_*\equiv R_*(\mu,\beta,[b]_{0,\beta},K_0)\) is the radius determined in Step~6, so that the $G$-phase estimates of Step~6 are available for every ball $B_i$ with $\rho_i\le r$. With these choices,
\eqref{small_est} is satisfied. Substituting \eqref{const_est_1}--\eqref{const_est_3} into \eqref{final_est} and recalling the definition of \(\lambda_0\) in
\eqref{step1-lambda0}, we infer
\begin{align*}
&\dashint_{B_{r_1}}\Theta'([\Psi_1(x,Du)]_t)\Psi_1(x,Du)\,dx\nonum\\
&\le\frac12\dashint_{B_{r_2}}\Theta'([\Psi_1(x,Du)]_t)\Psi_1(x,Du)\,dx+c(s(\Theta))\dashint_{B_R}\Theta(\Psi_1(x,F))\,dx\nonumber\\
&\qquad+\frac{c^{\,s(\Theta)+1}R^{n(s(\Theta)+1)}}
{(r_2-r_1)^{n(s(\Theta)+1)}}\Theta\!\left(\dashint_{B_R}
\bigl[\Psi_1(x,Du)+M\Psi_1(x,F)\bigr]\,dx\right).
\end{align*} 
where \(c\equiv c(\texttt{data}_0)\), \(c(s(\Theta))\equiv c(\texttt{data}_0,s(\Theta))\) and \(M\equiv M(\texttt{data}_0,s(\Theta))\).
The remaining part of the proof follows exactly as in Step~10 of \cite[Proof of Theorem~1.1]{Baasandorj2020}. Therefore, the proof of the main theorem is complete. \qed

\vspace{0.5cm}
\textbf{Funding} This work was supported by the National Research Foundation of Korea (NRF) grant funded by the Korea government [grant numbers RS-2025-00555316 and RS-2025-25415411].

\vspace{0.2cm}
\textbf{Conflict of interest} The authors declare that there is no conflict of interest.

\vspace{0.2cm}
\textbf{Data availability} No data was used for the research described in the article.

\end{document}